\documentclass[11pt,reqno]{amsart}

\usepackage[pagebackref]{hyperref}
\usepackage{amsmath,amssymb,amsthm,mathrsfs,cite}
\usepackage{enumitem}
\usepackage{microtype}
\theoremstyle{plain}

\newtheorem{theorem}{Theorem}[section]
\newtheorem{lemma}[theorem]{Lemma}
\newtheorem{proposition}[theorem]{Proposition}
\newtheorem{corollary}[theorem]{Corollary}

\theoremstyle{definition}
\newtheorem{defn}[theorem]{Definition}
\newtheorem{example}[theorem]{Example}
\newtheorem{remark}[theorem]{Remark}

\numberwithin{equation}{section}
\newcommand{\norm}[1]{\left\lVert#1\right\rVert}

\newcommand{\supp}{\mathrm{supp}}

\newcommand{\radius}{r}

\allowdisplaybreaks

\begin{document}

\title[Spectral Invariance and Gevrey Regularity]{Spectral Invariance and Gevrey Regularity for Groups with strongly subexponential growth}

\author{Lin Chen$^*$}
\address{Department of Mathematics and Statistics, Suzhou University of Technology,  Suzhou, $215500$, P. R. China}
\email{linchen198112@163.com}

\author{Qin Wang}
\address{Research Center for Operator Algebras, School of Mathematical Sciences, East China Normal University,
Shanghai, $200241$, P. R. China}
\email{qwang@math.ecnu.edu.cn}

\thanks{$^*$Corresponding author: Lin Chen}

%    General info


\begin{abstract}
We study spectral invariance and Gevrey regularity for convolution
operators with kernels in suitable weighted function spaces on locally
compact groups equipped with a locally bounded length function $\ell$.  The main analytic scale is given by the
subexponential weights
	$$\omega_s(x)=\exp(s\ell(x)^\beta),\qquad 0<\beta<1.$$
For groups whose volume growth is bounded above by $e^{R^\gamma}$ for
some $0<\gamma<1$, we establish spectral comparison result for
compactly supported functions.  For compactly supported Hermitian
functions, we prove spectral radius invariance across the symmetric
$q$-pseudofunction $*$-algebra, the weighted and unweighted group
algebras, and the full and reduced group $C^*$-algebras.
For unimodular groups satisfying strong subexponential growth of exponent
at most $\beta$, we construct a Gevrey-Beurling operator algebra inside
the unitized $q$-pseudofunction algebra. We prove that this algebra is
inverse-closed and that its inclusion induces an isomorphism in
topological $K$-theory. The inverse-closedness theorem may be viewed as a
quantitative Gevrey-type noncommutative Wiener lemma. As an application,
we show that whenever a convolution operators with kernels in the
corresponding weighted  Gevrey-Beurling space  is invertible in the unitized
$q$-pseudofunction algebra, then its inverse belongs to the same
Gevrey-Beurling operator algebra and satisfies explicit Gevrey seminorm
estimates. We also develop a relative theory for pairs of finitely generated
groups using Schreier graph lengths and quasi-regular representations.
This provides a subexponential analogue of rapid decay for group pairs,
when subgroup  is normal, it reduces to the usual theory on the
quotient.  The framework can apply to intermediate-growth examples, including the
Grigorchuk group, and is stable under products with polynomial growth
groups and under compact extensions.
\end{abstract}

\maketitle

\section{Introduction}

Spectral invariance is a central theme in noncommutative harmonic
analysis and noncommutative geometry. If $\mathcal A$ is a dense unital subalgebra of a unital Banach algebra
$\mathcal B$, sharing the same identity, and if $\mathcal A$ is
spectrally invariant in $\mathcal B$, then every element of
$\mathcal A$ that is invertible in $\mathcal B$ has its inverse again
in $\mathcal A$. Under standard Fréchet algebra hypotheses, spectral
invariance is equivalent to stability
under holomorphic functional calculus \cite{Schweitzer}. Such
subalgebras therefore preserve topological $K$-theory and serve as
natural smooth domains for cyclic cocycles and index-theoretic
constructions.

The purpose of this paper is to study Gevrey-type spectral regularity
for convolution operators with kernels in suitable weighted function spaces on locally compact groups satisfying strong
subexponential growth conditions. We say that $G$ has
\emph{strong subexponential growth of exponent} $0<\gamma<1$ if, for
every $c>0$, there exists a constant $C_c>0$ such that
$$\mu(B_R)\leq C_c\,e^{cR^\gamma}, \qquad R\geq 0.$$
Throughout the paper, the relevant analytic scale is provided by the
subexponential weights
$$\omega_s(x)=e^{s\ell(x)^\beta},\qquad 0<\beta<1,$$
where $\ell$ is the length function on $G$.

Classically, given an open set $\Omega\subseteq\mathbb{R}^d$ and an
exponent $\sigma\geq 1$, a smooth function $u\in C^\infty(\Omega)$ is
said to belong to the \emph{Gevrey class of order $\sigma$} if, for
every compact set $K\subset\Omega$, there exist constants $C,R>0$ such
that
$$\sup_{x\in K}|\partial^\alpha u(x)|\leq C R^{|\alpha|}(\alpha!)^\sigma,\qquad \alpha\in\mathbb{N}^d.$$
These classes were introduced by Maurice Gevrey in his 1918 work on
partial differential equations \cite{Gev}.
In the Euclidean setting, factorial bounds on derivatives are reflected
in subexponential decay of the Fourier transform, with the decay rate
governed by the reciprocal of the Gevrey order; see, for example,
\cite[Theorem~1.6.1]{Ro}. On compact Lie groups, the
Gevrey-Beurling class $\gamma_{(s)}(G)$ admits a characterization in
terms of subexponential decay of the Fourier transform, specifically,  a
function $\phi$ belongs to $\gamma_{(s)}(G)$ if and only if, for every
$B>0$, there exists a constant $K_B>0$ such that
$$\|\widehat{\phi}(\xi)\|_{\mathrm{HS}}\leq K_B\,e^{-B\langle\xi\rangle^{1/s}},
\qquad [\xi]\in\widehat{G}.$$
Here $\widehat{G}$ denotes the unitary dual of $G$,
$\|\cdot\|_{\mathrm{HS}}$ is the Hilbert--Schmidt norm, and
$\langle\xi\rangle$ is the eigenvalues of the elliptic first-order pseudo-differential operator; see
\cite[Theorem~2.3]{DasguptaRuzhansky2014}. Thus Gevrey regularity may
be viewed as a quantitative intermediate decay condition, it is faster
than any polynomial decay and, for orders strictly greater than one,
weaker than exponential decay.

This perspective also has an abstract Banach-algebraic analogue. In a
Banach algebra equipped with a closed derivation, ultradifferentiable and
Gevrey regularity can be formulated through growth conditions on the
iterates of the derivation. This view point is closely related to the
Dales-Davie Banach algebra constructions of \cite{DalesDavie1973}
and to the theory of inverse-closed smooth Banach subalgebras; see
\cite{Gr,Grk1,Gr2,Gr3,Gr4,Grk5,Klo,Bla} and the references therein.

In the locally compact group setting, the abstract closed derivation is
replaced by the commutator derivation associated with the length
function $\ell$. Setting $D_\ell$ to be the multiplication operator by
$\ell$, we define $\delta_\ell(T)=[D_\ell,T]$.
For $0<\beta<1$, the  Gevrey-Beurling type space
$$S_q^\infty(G)=\bigcap_{s>0}L^q\!\bigl(G,e^{s\ell^\beta}\bigr)$$
consists of functions decaying faster than $e^{-s\ell^\beta}$ for
every $s>0$. This spatial decay can translates into operator-theoretic
Gevrey regularity, for convolution operator $T_f$ associated with a
kernel $f\in S_q^\infty(G)$, the iterates of
$\delta_\ell$ satisfy estimates of Gevrey order $1/\beta$,
$$\sup_{k\geq 0}\frac{R^k}{(k!)^{1/\beta}}\bigl\|\delta_\ell^k(T_f)\bigr\|_{B(L^q(G))}<\infty
\qquad\text{for every }R>0.$$
Thus the natural Gevrey order associated with the scale
$e^{s\ell^\beta}$ is $\sigma=1/\beta$; see Proposition~\ref{prochen1}.
This observation links subexponential decay on the group with
smoothness of convolution operators.  Thus the Gevrey-Beurling
opertor algebra constructed below may be viewed as subexponential counterparts
of the Sobolev-type smooth algebras arising in the rapid decay setting.
The exponent $\beta$ controls the subexponential decay scale, while its
reciprocal determines the corresponding Gevrey order.

This approach is complementary to the rapid decay theory of groups.
Classical rapid decay employs polynomial weights to construct
Sobolev-type smooth subalgebras of group $C^*$-algebras
\cite{Jo1,Jo2}, which play a fundamental role in noncommutative
geometry---notably in the work of Connes--Moscovici
\cite{ConnesMoscovici1990} and Lafforgue \cite{Laff}. However, for
amenable groups, rapid decay is equivalent to polynomial growth, so
groups of intermediate growth lie outside the polynomial rapid decay
framework. Subexponential weights provide a finer analytic scale
naturally adapted to such groups.

\medskip
Our first main result is a spectral comparison theorem.

{\it{ \noindent \bf Theorem A} (cf. Theorem~\ref{main1}).
Let $G$ be a compactly generated unimodular locally compact group, and
let $\ell$ be a length function equivalent to a word length associated
with a compact generating set. Let
$\omega_s(x)=e^{s\ell(x)^\beta}$
be a subexponential weight with exponent $0<\beta<1$
(see Definition~\ref{def1}). Assume that the volume growth of $G$ with
respect to $\ell$ is bounded by a subexponential function of order
$\gamma$, for some $0<\gamma<1$; that is, there exist constants
$C,c>0$ such that
$$\mu(B_R)\leq C e^{cR^\gamma},\qquad R\geq 0,$$
where
$B_R=\{x\in G:\ell(x)\leq R\}.$
Then, for every $f\in C_c(G)$,
$$\sigma_{1,s}(f)=\sigma_{L^1(G)}(f)=\sigma_{C_r^*(G)}(f)=\sigma_{C^*(G)}(f).$$
If, in addition, $f=f^*$, then for every $1<q<\infty$,
$$r_{1,s}(f)=r_{L^1(G)}(f)=r_{PF_q^*(G)}(f)=r_{C_r^*(G)}(f)=r_{C^*(G)}(f).$$
In particular, $L^1(G,\omega_s)$ is quasi-symmetric.}

\medskip

The classical Wiener lemma states that the algebra of absolutely
convergent Fourier series is inverse-closed in $C(\mathbb T)$. In other
words, if such a series defines a nowhere vanishing continuous function
on the circle, then its reciprocal again has an absolutely convergent
Fourier series. Our second main result is a quantitative Gevrey-type
noncommutative Wiener lemma.
\medskip

{\it {\noindent \bf Theorem B} (cf. Theorem~\ref{inverse-closed}).
Let $G$ be a locally compact group equipped with a proper locally bounded
length function $\ell$, and let $\delta_\ell$ be the derivation defined
in Definition~\ref{def45}. For $1<q<\infty$ and $0<\beta<1$, let
$\widetilde{PF_q(G)}$ denote the unitization of $PF_q(G)$, and set
$$G^{(1/\beta)}(B(L^q(G)))=\bigcap_{R>0}G_R^{(1/\beta)}(B(L^q(G))),$$
the operator-algebraic Gevrey-Beurling class (see Definition \ref{def23}) associated with
$\delta_\ell$. Then the algebra
$$\mathcal A_{q,\beta}(G)=\widetilde{PF_q(G)}\cap G^{(1/\beta)}(B(L^q(G)))$$
is inverse-closed in $\widetilde{PF_q(G)}$.
Moreover, for every $R>0$ and every $0<R'<R$, the inverse satisfies the
quantitative estimate
$$\|T^{-1}\|_{1/\beta,R'}\leq\|T^{-1}\|_{B(L^q(G))}
v_{1/\beta-1}\left(\|T^{-1}\|_{B(L^q(G))}\frac{\|T\|_{1/\beta,R}}{1-R'/R}\right),$$
where
$v_{1/\beta-1}(x)=\sum_{m=0}^{\infty}\frac{x^m}{(m!)^{1/\beta-1}}$.}

If $G$ is unimodular and satisfies condition $(SG_\beta)$, meaning that
$G$ has strong subexponential growth with some exponent not exceeding
$\beta$, where $0<\beta<1$ is the exponent of the subexponential weight,
then the inclusion
$\mathcal A_{q,\beta}(G)
\hookrightarrow
\widetilde{PF_q(G)}$
induces an isomorphism in topological $K$-theory; see
Corollary~\ref{K}.
Combining the preceding results with the Gevrey bounds for convolution
operators with kernels in the weighted Gevrey-Beurling space
$S_q^\infty(G)$( see Proposition \ref{prochen1}), we obtain the following application.
\medskip

{\it {\noindent \bf Theorem C} (cf. Corollary~\ref{main2}).
Assume that $G$ is unimodular and satisfies condition $(SG_\beta)$ with
respect to a proper locally bounded length function $\ell$, where
$0<\beta<1$. Let $1<q<\infty$. For $f\in S_q^\infty(G)$, let
$T_f:=\lambda_q(f)$
be the associated convolution operator on $L^q(G)$. If $T_f$ is
invertible in the unitization $\widetilde{PF_q(G)}$, then
$T_f^{-1}\in\mathcal A_{q,\beta}(G)$,
where the Gevrey-Beurling class is taken with respect to the derivation
$\delta_\ell$. Moreover, the Gevrey seminorms of $T_f^{-1}$ satisfy the
estimate of Theorem~B.}
\medskip

Our third main result is a relative version for pairs of finitely
generated groups. Using the quotient length on the Schreier graph and
the quasi-regular representation, we obtain a subexponential analogue of
pair rapid decay in \cite{ChatterjiZarka2024}.
\medskip

{\it {\noindent \bf Theorem D} (cf. Theorem~\ref{thm:relative-gevrey-wiener}).
Let $G$ be a finitely generated discrete group and let $H\leq G$ be a
finitely generated subgroup. Let $0<\beta<1$ and $1<q<\infty$. If
$$T\in\mathcal A_{q,\beta}(G,H)=\widetilde{PF_q^H(G)}\cap G_H^{(1/\beta)}(B(\ell^q(G/H)))$$
and $T$ is invertible in $\widetilde{PF_q^H(G)}$, then
$T^{-1}\in\mathcal A_{q,\beta}(G,H)$.
In particular, if $f\in\mathcal S_{1,H}^{\infty}(G)$ and
$\pi_q^H(f)$ is invertible in $\widetilde{PF_q^H(G)}$, then
$$\pi_q^H(f)^{-1}\in\widetilde{PF_q^H(G)}\cap G_H^{(1/\beta)}(B(\ell^q(G/H))).$$
Moreover, the Gevrey seminorms of the inverse satisfy the estimates in
Theorem B.}

When the subgroup $H$ is normal, this construction reduces to the
ordinary Gevrey theory on the quotient group; see
Corollary~\ref{normal}. The relative condition can be strictly weaker
than the corresponding global Gevrey condition. For example, see
Example~\ref{ex612}: if
$$G=F_2\times\mathbb Z^d\qquad\text{and}\qquad H=F_2\times\{0\},$$
then $G/H\cong\mathbb Z^d$.
Thus the relative length only detects the quotient direction, whereas
the global length also detects the $F_2$-direction. Consequently,
Proposition~\ref{prop:pair-gevrey-embedding} shows that such kernels
give rise to relative Gevrey operators, and
Theorem~\ref{thm:relative-gevrey-wiener} implies that invertibility in
the unitized $q$-pseudofunction algebra preserves the corresponding
relative Gevrey regularity.
\medskip

Finally, we present examples showing that the framework applies to
groups of intermediate growth and strong subexponential growth,
including examples involving the Grigorchuk group, direct products with
groups of polynomial growth, and compact extensions. These examples
illustrate that Gevrey regularity provides a natural and effective
replacement for polynomial rapid decay when subexponential estimates are
the appropriate analytic control.
\medskip

The paper is organized as follows. 
Section~2 recalls the basic definitions and preliminary results.
Section~3 proves the spectral comparison theorem. Section~4 establishes norm estimates for iterated commutators of
convolution operators and derives Gevrey bounds of order $1/\beta$. It
then constructs the associated Gevrey-Beurling operator algebra and
proves a quantitative Gevrey-type noncommutative Wiener lemma. As an
application, we show that if a convolution operator with kernel in the
corresponding weighted Gevrey-Beurling space is invertible in the
unitized $q$-pseudofunction algebra, then its inverse belongs to the same
Gevrey-Beurling operator algebra and satisfies explicit Gevrey seminorm
estimates.
Section~5 presents examples and stability results. Section~6 develops
relative Gevrey regularity. Section~7 discusses open problems.

\section{Preliminaries}
In this section, we recall some basic terminology. Let $G$ be a locally compact group with identity element $e$. A function $\ell:G\to \mathbb R_+$ is called a \emph {length function} if
$$\ell(e)=0,\qquad\ell(gh)\leq \ell(g)+\ell(h),\qquad\ell(g^{-1})=\ell(g)$$
for all $g,h\in G$. We say that $\ell$ is \emph {locally bounded} if it is bounded
on compact subsets of $G$, and \emph {proper} if $\ell^{-1}([0,c])$
is compact for every $c\geq 0$.
\begin{defn}\label{def1}
Let $G$ be a locally compact group.  A \emph{weight} on $G$ is a
Borel measurable function $\omega: G\rightarrow[1,\infty)$ such that $\omega(e)=1$ and
$$\omega(xy)\leq\omega(x)\omega(y)$$ for all $x,y\in G$. The weight $\omega$ is said to be \emph{symmetric} if $\omega(x)=\omega(x^{-1})$.
For any $s>0$ and $0<\beta< 1$,  we define the weight function $\omega_s: G\to[1, \infty)$ by $\omega_s(x)=e^{s\ell(x)^\beta}$.
Since the function $t\mapsto t^\beta$ is concave and subadditive on $[0, \infty)$, we have
$$\ell(xy)^\beta\le(\ell(x)+\ell(y))^\beta\le\ell(x)^\beta+\ell(y)^\beta.$$
This implies the submultiplicativity of the weight function
$$\omega_s(xy)=e^{s\ell(xy)^\beta}\le e^{s\ell(x)^\beta} e^{s\ell(y)^\beta}=\omega_s(x)\omega_s(y).$$
Since the length function is symmetric, $\omega_s$
is symmetric as well. We call such a symmetric weight \emph{subexponential weight   with the exponent $\beta$}.
\end{defn}
\begin{defn}
Let $G$ be a locally compact group equipped with a locally bounded
length function $\ell$, and let $\mu$ be a Haar measure on $G$.
For $0<\gamma<1$, we say that $G$ has \emph {strong subexponential growth of
exponent $\gamma$} with respect to $\ell$ if, for every $c>0$, there
exists a constant $C_c>0$ such that
$$\mu(B_R)\le C_c e^{cR^\gamma},\qquad R\ge 0.$$
\end{defn}

Let $G$ be a locally compact unimodular group with a left invariant
Haar measure $\mu$ on $G$. Now let $C_c(G)$ denote the space of compactly supported continuous functions on $G$. We endow  $C_c(G)$ with the
convolution product, which for $f,g\in C_c(G)$ is given by  
$$(f\ast g)(x)=\int_G f(y) g\bigl(y^{-1}x\bigr)d\mu(y)$$ 
for $x\in G$, and involution given by
$f^*(x)=\overline{f(x^{-1})}$,
for $f\in C_c(G)$ and $x\in G$.
For  $1<q<\infty$. The \emph{left-regular representation} of
$C_c(G)$ on $L^q(G)$ is given by
$$\lambda_q: C_c(G)\to B(L^q(G)), \qquad \lambda_q(f)g=f \ast g,$$
for all $f\in C_c(G)$ and $g\in L^q(G)$.
The norm-closure of $\lambda_q(C_c(G))$ inside $B(L^q(G))$ is denoted by
$PF_q(G)$ and called \emph {the algebra of $q$-pseudofunctions on $G$} (see \cite[Section 4]{Der} for references
and further information). Clearly we have $PF_2(G)=C_r^*(G)$, the reduced $C^*$ algebra of $G$. Suppose $1\le q\le p\le\infty$ are conjugate, i.e.
$\frac{1}{q}+\frac{1}{p}=1$,
consider the duality relation $L^q(G)^*\cong L^p(G)$ given by
$$\langle f,g\rangle:=\int_G f(s)\overline{g(s)}ds\qquad(f\in L^q(G),\ g\in L^p(G)).$$
Then, for every $f\in C_c(G)$, $g\in L^q(G)$ and $h\in L^p(G)$, it is straightforward to verify that
$\langle\lambda_q(f^*)g, h\rangle=\langle g, \lambda_p(f)h\rangle.$
The Banach $*$-algebra $PF_q^*(G)$ is defined to be the completion of
$\bigl(C_c(G), \|\cdot\|_{PF_q^*(G)}, *\bigr)$.
with the standard convolution and involution,
and
$$\|f\|_{PF_q^*(G)}:=\max\Bigl\{\|\lambda_q(f)\|_{B(L^q(G))},\|\lambda_p(f)\|_{B(L^p(G))}Bigr\}.$$
It is immediate that if $q$ is the conjugate of $p$, then $PF_p^*(G)=PF_q^*(G)$ isometrically as
Banach $*$-algebras. Samei and Wiersma (\cite{Sa}) explored these algebras in relation to group amenability, while Liao and Yu (\cite{Liao}) studied them within the framework of
$K$-theory.

We now introduce the  Gevrey-Beurling  class and the Dales-Davie algebra.  Historically, Dales--Davie algebras were introduced by Dales and Davie in the scalar Banach function algebra setting \cite{DalesDavie1973}, they were later used in the abstract Banach-algebraic theory of ultradifferentiable and Gevrey-type inverse-closed subalgebras \cite{Klo}.  Classical scalar counterparts in the complex plane are discussed in \cite{Sidd}.

\begin{defn}\label{def23}
Let $\mathcal{A}$ be a unital Banach algebra and let $\delta$ be a closed derivation on $\mathcal{A}$.  We use the convention
$\delta^0=\mathrm{id}_{\mathcal A}$ and define 
$$\mathcal{A}^\infty=\bigcap_{n=1}^\infty \mathrm{Dom}(\delta^n)$$
to be the smooth subalgebra associated with the closed derivation $\delta$. For $\sigma > 1$ and $R>0$, we define  
$$G_{R}^{(\sigma)}(\mathcal{A})
=\left\{T\in \mathcal{A}^\infty:\ \|T\|_{\sigma,R}:=
\sup_{n\ge 0}\frac{R^n\|\delta^n(T)\|_{\mathcal{A}}}{(n!)^\sigma}<\infty\right\}.$$
The \emph {operator-algebraic Gevrey-Beurling class $G^{(\sigma)}(\mathcal{A})$  of order $\sigma$} associated with $\delta$ is
defined by
$$G^{(\sigma)}(\mathcal{A})=\bigcap_{R>0}G_{R}^{(\sigma)}(\mathcal{A})$$
equipped 
with the projective limit topology generated by the seminorms $\{\|\cdot\|_{\sigma,R}\}_{R>0}$. 
\end{defn}

\begin{remark}\label{rek24}
With this topology, $G^{(\sigma)}(\mathcal A)$ is a Fr\'echet algebra.
Indeed, for each $R>0$, the closedness of $\delta$ implies that each
$\delta^n$ is closed, and The Closed Graph Theorem shows that
$G_R^{(\sigma)}(\mathcal A)$ is complete with respect to the seminorm
$\|\cdot\|_{\sigma,R}$.
Since the seminorms are increasing in $R$, the projective limit topology
is generated by the countable family
	$\{\|\cdot\|_{\sigma,m}\}_{m\in\mathbb N}$. Hence
$G^{(\sigma)}(\mathcal A)$ is a Fr\'echet space.
It remains to check that multiplication is continuous. Let $0<R<R^\prime$.
For $S,T\in G^{(\sigma)}(\mathcal A)$, the iterated Leibniz formula gives
$$\delta^n(ST)=\sum_{k=0}^n\binom{n}{k}\delta^k(S)\delta^{n-k}(T).$$
Hence
$$\|\delta^n(ST)\|_{\mathcal A}\le\|S\|_{\sigma,R^\prime}\|T\|_{\sigma,R^\prime}
\frac{1}{(R^\prime)^n}\sum_{k=0}^n\binom{n}{k}(k!)^\sigma((n-k)!)^\sigma.$$
Since $\sigma>1$, we have
$$\binom{n}{k}(k!)^\sigma((n-k)!)^\sigma=n!(k!)^{\sigma-1}((n-k)!)^{\sigma-1}	\le(n!)^\sigma.$$
Therefore
$$\frac{R^n\|\delta^n(ST)\|_{\mathcal A}}{(n!)^\sigma}
\le(n+1)\left(\frac{R}{R^\prime}\right)^n\|S\|_{\sigma,R^\prime}\|T\|_{\sigma,R^\prime}.$$
Taking the supremum over $n\ge0$, we obtain
$$\|ST\|_{\sigma,R}\le C_{R,R^\prime}\|S\|_{\sigma,R^\prime}\|T\|_{\sigma,R^\prime},$$
where
$$C_{R,R^\prime}=\sup_{n\ge0}(n+1)\left(\frac{R}{R^\prime}\right)^n<\infty.$$
Thus multiplication is jointly continuous in the projective limit
topology, and $G^{(\sigma)}(\mathcal A)$ is a Fr\'echet algebra.
\end{remark}

While the Gevrey-Beurling seminorms impose a supremum-type bound on the
growth of derivatives, it is often technically useful to work with a
slightly stronger $\ell^1$-type summability condition. This motivates the
following Dales-Davie  algebra.

\begin{defn}
For $R>0$, define the Dales-Davie algebra by
$$\mathcal{D}^1_{M_R}=\left\{T\in\mathcal A^\infty:
\|T\|_{\mathcal D^1_{M_R}}:=\sum_{n=0}^\infty
\frac{R^n}{(n!)^\sigma}\|\delta^n(T)\|_{\mathcal A}<\infty\right\}.$$
\end{defn}
\begin{remark}
The space $\mathcal D^1_{M_R}$ is a Banach algebra. Completeness follows
from the closedness of $\delta$ by The Closed Graph Theorem  argument. Moreover,
for $S,T\in\mathcal D^1_{M_R}$, the Leibniz formula gives
$$\delta^n(ST)=\sum_{k=0}^n\binom{n}{k}\delta^k(S)\delta^{n-k}(T).$$
Hence
$$\|ST\|_{\mathcal D^1_{M_R}}\leq\sum_{k,l\geq0}
\binom{k+l}{k}^{1-\sigma}\frac{R^k\|\delta^k(S)\|_{\mathcal A}}{(k!)^\sigma}
\frac{R^l\|\delta^l(T)\|_{\mathcal A}}{(l!)^\sigma}.$$
Since $\sigma>1$, we have
$$\binom{k+l}{k}^{1-\sigma}\leq1.$$
Therefore
$\|ST\|_{\mathcal D^1_{M_R}}\leq\|S\|_{\mathcal D^1_{M_R}}\|T\|_{\mathcal D^1_{M_R}}$.
Thus $\mathcal D^1_{M_R}$ is a Banach algebra.
\end{remark}

In the setting of locally compact groups, we model Gevrey-type
regularity at the level of functions by imposing subexponential weighted
$L^q$-decay with respect to a length function. This gives a weighted
length-function analogue of the classical Gevrey-Beurling classes
(see \cite[Definition~2.2]{DasguptaRuzhansky2014}).

\begin{defn}
Let $G$ be a locally compact group equipped with a locally bounded
length function $\ell$, and let $0<\beta<1$. For $s>0$ and
$1<q<\infty$, set
$$\|f\|_{q,s}=
\left(\int_G |f(x)|^q e^{sq\ell(x)^\beta}\,d\mu(x)\right)^{1/q},$$
and write
$$L^q(G,e^{s\ell^\beta})=\{f\in L^q_{\mathrm{loc}}(G):\|f\|_{q,s}<\infty\}.$$
We define \emph{the weighted Gevrey-Beurling type space} on $G$  by
$$S_q^{\infty}(G)=\bigcap_{s>0}L^q(G,e^{s\ell^\beta}).$$
\end{defn}

\section{Spectral invariance on compact support functions}
Recall that a Banach $*$-algebra $\mathcal A$ is said to be {\it symmetric} if for all $a\in \mathcal A$ we have
$\sigma_{\mathcal A}(a^*a)\subset[0,\infty)$, this is equivalent to saying that $\mathcal A$ is {\it Hermitian} ($a=a^*\in\mathcal A$ implies $\sigma_{\mathcal A}(a)\subset \mathbb R$). A locally compact group $G$ is {\it Hermitian} if $L^1(G)$
is a Hermitian Banach $*$-algebra. The symmetry of group algebras and weighted group algebras has a long
history.  Poguntke proved symmetry for connected nilpotent Lie groups
\cite{Pog}; Pytlik obtained symmetry for suitable polynomially weighted
algebras on connected groups of polynomial growth \cite{Pyt}; and Losert
proved symmetry for subexponentially weighted algebras on compactly
generated groups of polynomial growth \cite{los1,los2}.  More recently,
quasi-Hermitian groups were introduced in \cite{pal}, where it was shown
that locally compact groups of subexponential growth are quasi-Hermitian.
Spectral interpolation methods for triples of Banach $*$-algebras were later developed in \cite{Sa}, with applications to
amenability. 

Motivated by these results, in this section we establish a spectral
invariance theorem for subexponentially weighted group algebras on
groups of strong subexponential growth. we prove that every compactly supported function has the same spectrum
in the weighted group algebra $L^1(G,\omega_s)$, in the ordinary group
algebra $L^1(G)$, and in both the reduced and full group $C^*$-algebras.
For self-adjoint compactly supported functions, we further obtain a
spectral-radius identity involving the $L^q$-pseudofunction algebra
$PF_q^*(G)$, for every $1<q<\infty$. 
This result may be viewed as a subexponential-growth analogue of the
known spectral invariance results for weighted group algebras on groups
of polynomial growth.

\begin{defn}
Let $\mathcal A$ be a Banach $*$-algebra associated with a locally compact group $G$, and suppose that $C_c(G)$ is a dense $*$-subalgebra of $\mathcal A$. We say that $\mathcal A$ is \emph{quasi-symmetric} if
$$\sigma_{\mathcal A}(f^* * f) \subseteq \mathbb{R}_{\ge 0}$$
for every $f \in C_c(G)$. A locally compact group $G$ is quasi-symmetric if
the group algebra $L^1(G)$ is a quasi-symmetric Banach $*$-algebra.
\end{defn}

To prove the main result of this section, we rely on the following celebrated theorem, which can be found in \cite[Proposition 2.5]{Hu}, \cite[Theorem 2.3]{Sa} and \cite[Proposition 6.1]{Fend2}.
\begin{theorem}[Barnes-Hulanicki Theorem]\label{BHT}
Let $\mathcal A$ be a Banach $*$-algebra, $\mathcal S$ a $*$-subalgebra of $\mathcal A$, and
$\pi : \mathcal A \to B(\mathcal{H})$ a faithful $*$-representation.
If $A$ is unital, we assume that
$\pi(1_A)=\mathrm{id}_{B(\mathcal{H})}$.
If $$r_{\mathcal A}(a)=\|\pi(a)\|$$ for all $a \in \mathcal S_h$, then
$$\sigma_A(a)=\sigma_{B(\mathcal{H})}(\pi(a))$$
for every $a \in \mathcal S$.
\end{theorem}

The main result of this section is the following theorem.

\begin{theorem}\label{main1}
Let $G$ be a compactly generated unimodular locally compact group, and
let $\ell$ be a length function equivalent to a word length associated
with a compact generating set. Let
$\omega_s(x)=e^{s\ell(x)^\beta}$  be a subexponential weight with the exponent $\beta$ (Definition \ref{def1}).
Assume that the growth of $G$ with respect to $\ell$ is bounded by
a subexponential function of order $\gamma$, for some $0<\gamma<1$, that
is, there exist constants $C,c>0$ such that
$$\mu(B_R)\leq C e^{cR^\gamma},\qquad R\geq 0,$$
where $B_R=\{x\in G:\ell(x)\leq R\}$.
Then, for every $f\in C_c(G)$,
$$\sigma_{1,s}(f)=\sigma_{L^1(G)}(f)=\sigma_{C_r^*(G)}(f)=\sigma_{C^*(G)}(f).$$
If, in addition, $f=f^*$, then for every $1<q<\infty$,
$$r_{1,s}(f)=r_{L^1(G)}(f)=r_{PF_q^*(G)}(f)=r_{C_r^*(G)}(f)=r_{C^*(G)}(f).$$
In particular, $L^1(G,\omega_s)$ is quasi-symmetric.
\end{theorem}

\begin{proof}
Since $G$ is unimodular and  $\omega_s$ is
submultiplicative,  we know $L^1(G,\omega_s)$ is a Banach $*$-algebra,
Write $A_s=L^1(G, \omega_s)$ and let $C_c(G)_h$ denote the set of self-adjoint elements in $C_c(G)$. Since $\omega_s$ is bounded on compact sets, we have $C_c(G)\subset A_s$ and
$C_c(G)$ is dense in $A_s$ with respect to $\|\cdot\|_{1, s}$. Since $\omega_s(x) \ge 1$, $A_s$ can  continuously embeds into  $L^1(G)$.
For $f\in C_c(G)$, assume $\supp(f) \subset B_K$ for some $K > 0$.  By the subadditivity of the length function, the support of the $n$-th convolution power satisfies $\supp(f^{*n}) \subset B_{nK}$. By Cauchy-Schwarz inequality
\begin{align}
\norm{f^{*n}}_{1, s} 
&=\int_{B_{nK}}|f^{*n}(x)|\omega_s(x)d\mu(x)\nonumber \\
&\le\left(\int_{G}|f^{*n}(x)|^2 d\mu(x)\right)^{1/2}\left(\int_{B_{nK}}\omega_s(x)^2 \ d\mu(x)\right)^{1/2} \nonumber\\
&=\norm{f^{*n}}_{L^2(G)}\left(\int_{B_{nK}}e^{2s\ell(x)^\beta} d\mu(x)\right)^{1/2} \label{cleq5}.
\end{align}
Let $\lambda$ be the left regular representation of $G$ on $L^2(G)$.
The operator norm $\|\lambda(f)\|_{B(L^2(G))}$ bounded by the enveloping  $C^*$-norm $\norm{f}_{C^*(G)}$.  Since
$f^{\ast n}=f\ast(f^{\ast (n-1)})=\lambda(f)(f^{\ast (n-1)})$,
we have
$$\|f^{*n}\|_{L^2(G)}\le\|\lambda(f)\|_{B(L^2(G))}\|f^{*(n-1)}\|_{L^2(G)}.$$
Iterating this inequality $n-1$ times yields
\begin{equation}\label{cleq6}
\|f^{*n}\|_{L^2(G)}\le\|\lambda(f)\|_{B(L^2(G))}^{n-1}  \|f\|_{L^2(G)}\le\norm{f}_{C^*(G)}^{n-1}\|f\|_{L^2(G)}.
\end{equation}
Using the growth condition $\mu(B_R)\le Ce^{c R^\gamma}$ and the bound $\omega_s(x) \le e^{s(nK)^\beta}$ on $B_{nK}$, we have
$$\int_{B_{nK}}\omega_s(x)^2d\mu(x)\le\mu(B_{nK})\sup_{x \in B_{nK}}e^{2s\ell(x)^\beta}\le C e^{c(nK)^\gamma}e^{2sn^\beta K^\beta}.$$
Combining \eqref{cleq5} and \eqref{cleq6},
\begin{align*}
\radius_{A_s}(f)
&=\lim_{n\to\infty}\norm{f^{*n}}_{1, s}^{1/n}\\
&\le\lim_{n \to \infty}\norm{f}_{C^*(G)}^{\frac{n-1}{n}}\lim_{n \to \infty} \|f\|_{L^2(G)}^{\frac{1}{n}}\lim_{n \to \infty} \left( C^{1/2}e^{\frac{cK^\gamma}{2} n^\gamma}e^{s K^\beta n^\beta}\right)^{1/n}\\
&=\norm{f}_{C^*(G)}\lim_{n \to \infty}\exp\left(\frac{\ln C}{2n}+\frac{c K^\gamma}{2} n^{\gamma-1}+s K^\beta n^{\beta-1}\right).
\end{align*}
Since $0<\gamma<1$ and $0<\beta<1$, we have the limit of the exponential term is $1$. Therefore
\begin{equation}\label{cleq8}
\radius_{A_s}(f) \le \norm{f}_{C^*(G)}.
\end{equation}
for all $f\in C_c(G)$. On the other hand, since $\omega_s\ge 1$, the inclusion map
$$\iota_s:A_s\to L^1(G),\qquad \iota_s(f)=f$$
is a continuous $*$ homomorphism. Composing with the canonical
$*$ homomorphism $L^1(G)\to C^*(G)$ yields a continuous $*$-homomorphism
$\pi_s:A_s\to C^*(G)$.
Hence spectral inclusion gives $\sigma_{C^*(G)}(\pi_s(a))\subseteq \sigma_{A_s}(a)$
for all $a\in A_s$. Therefore
$$r_{A_s}(f)\ge r_{C^*(G)}(\pi_s(f))=\|\pi_s(f)\|_{C^*(G)}=\|f\|_{C^*(G)}$$
for all $f\in C_c(G)_h$. Combining this with \eqref{cleq8} yields
\begin{equation}\label{eq:equality_Cc}
r_{A_s}(f)=\|f\|_{C^*(G)}\qquad (f\in C_c(G)_h).
\end{equation}
Now we apply Barnes-Hulanicki Theorem \cite[Theorem 2.3]{Sa} with
$$\mathcal A=A_s, \qquad \mathcal S=C_c(G),$$
and with the faithful  $*$-representation
$$\pi :=\pi_u \circ\iota\circ\iota_s: A_s \xrightarrow {\ \iota_s\ }L^1(G)\xrightarrow{\ \iota\ } C^*(G)\xrightarrow{\ \pi_u\ } B(\mathcal H_u),$$
where $\pi_u: C^*(G)\to B(\mathcal H_u)$ denotes the universal
faithful representation of the full group $C^*$-algebra $C^*(G)$. The representation $\pi$ is faithful. Indeed, $\iota_s$ is injective by
definition. The canonical map $\iota:L^1(G)\to C^*(G)$ is also injective,
if $f\in L^1(G)$ has zero full $C^*$-norm, then its reduced norm is zero,
and hence $\lambda(f)=0$ on $L^2(G)$. Since the left regular
representation of $L^1(G)$ is faithful, it follows that $f=0$. Finally,
$\pi_u$ is faithful by construction. Thus $\pi$ is faithful.
Since $C_c(G)$ is a dense $*$-subalgebra of $A_s$ and by Eq. \eqref{eq:equality_Cc}
$$r_{A_s}(f)=\norm{f}_{C^*(G)}=\norm {\pi(f)}_{B(\mathcal H_u)}, \qquad f\in\mathcal S_h,$$
the hypotheses of the Barnes--Hulanicki theorem are satisfied. Hence
$$\sigma_{A_s}(f)=\sigma_{B(H_u)}(\pi(f))$$ for all $f\in C_c(G)$.
Since $\pi_u$ is faithful, $\sigma_{B(H_u)}(\pi(f)) =\sigma_{C^*(G)}(f)$, and therefore
\begin{equation}\label{eqk1}
\sigma_{A_s}(f)=\sigma_{C^*(G)}(f)\qquad (f\in C_c(G)).
\end{equation}
Since $G$ is of subexponential growth, by \cite[Proposition 3]{pal}, $G$ is quasi-symmetric. Therefore, by
\cite[Theorem 4.7]{Sa},
\begin{equation}\label{eqk2}
\sigma_{L^1(G)}(f)=\sigma_{C_r^*(G)}(f)\qquad (f\in C_c(G)).
\end{equation}
Moreover, $G$ is amenable (see \cite[Corollary 4.8]{Sa}). Hence
$C^*(G)\cong C_r^*(G)$ canonically, and therefore
\begin{equation}\label{eqk5}
	\sigma_{C_r^*(G)}(f)=\sigma_{C^*(G)}(f)\qquad (f\in C_c(G)).
\end{equation}
Combining Eq. \eqref{eqk1}, Eq. \eqref{eqk2} and Eq. \eqref{eqk5}, gives
$$\sigma_{1,s}(f)=\sigma_{L^1(G)}(f)=\sigma_{C_r^*(G)}(f)=\sigma_{C^*(G)}(f)$$
for all $f\in C_c(G)$.
For the $PF_q^*(G)$ part, let $f\in C_c(G)_h$. If $1<q<2$, then by
\cite[Corollary 4.6]{Sa}, the triple
$$(L^1(G),PF_q^*(G),C_r^*(G))$$
is a spectral interpolation triple. Hence by \cite[Theorem 3.4]{Sa},
\begin{equation}\label{eqk4}
	r_{PF_q^*(G)}(f)=r_{C_r^*(G)}(f)\qquad (f\in C_c(G)_h).
\end{equation}
The case $q=2$ is immediate. If $q>2$, let $p$ be the conjugate exponent
of $q$. Then $1<p<2$, and by the standard duality identification
$$PF_p^*(G)=PF_q^*(G)$$
isometrically as Banach $*$-algebras. Therefore Eq. \eqref{eqk4} holds
for all $1<q<\infty$. This proves
$$r_{1,s}(f)=r_{L^1(G)}(f)=r_{PF_q^*(G)}(f)=r_{C_r^*(G)}(f)=r_{C^*(G)}(f)$$ 
for all $f\in C_c(G)_h$. Finally, the spectrum of a self-adjoint element in a $C^*$-algebra is real. The spectral identities above therefore imply that every
self-adjoint element of $C_c(G)$ has real spectrum in both
$L^1(G,\omega_s)$.  Hence $L^1(G,\omega_s)$  is quasi-symmetric.
\end{proof}

\begin{remark}
The theorem does not assert that
$\sigma_{PF_q^*(G)}(f)=\sigma_{C_r^*(G)}(f)$
for arbitrary $f\in C_c(G)$. Indeed, interpolation gives a canonical
continuous $*$-homomorphism
$$\Phi_q:PF_q^*(G)\to C_r^*(G),$$
but this map is not known to be injective in general. Hence the
spectral-radius identity for self-adjoint elements of $C_c(G)$ does
not by itself identify the spectra in $PF_q^*(G)$ and $C_r^*(G)$.
If, in addition, the canonical homomorphism $\Phi_q$ is injective and
the hypotheses of the Barnes--Hulanicki theorem are satisfied, then we obtain
	$$\sigma_{PF_q^*(G)}(f)=\sigma_{C_r^*(G)}(f)\qquad(f\in C_c(G)).$$
Consequently, under this additional injectivity assumption, for all  $f\in C_c(G)$ we get 
$$\sigma_{1,s}(f)=\sigma_{L^1(G)}(f)=\sigma_{PF_q^*(G)}(f)=\sigma_{C_r^*(G)}(f)=\sigma_{C^*(G)}(f).$$
Thus, these Banach $*$-algebras are quasi-symmetric.
\end{remark}

\section{Gevrey regularity}

In this section we construct a Gevrey-Beurling-type smooth subalgebra of
the algebra of $q$-pseudofunctions and prove that it is stable under
inversion and holomorphic functional calculus. The resulting
inverse-closedness theorem may be viewed as a Gevrey-type noncommutative
Wiener lemma. We then apply this result to convolution operators with
kernels in Gevrey-Beurling type space and obtain quantitative Gevrey regularity estimates for their inverses. Our approach
is inspired by the general principle, developed by Gr\"ochenig and Klotz
in \cite{Grk1,Grk5}, that suitable smoothness conditions on a Banach
subalgebra lead to inverse-closedness and norm-controlled inversion.

In our setting, smoothness is measured by iterated commutators with the
multiplication operator associated with a length function $\ell$. The
resulting estimates show that the $k$-th commutator grows at most like
$(k!)^{1/\beta}$, which is precisely of Gevrey type. The analytic input
is the subexponential weight
$$\omega_s(x)=e^{s\ell(x)^\beta},\qquad s>0,\qquad 0<\beta<1,$$
together with the strong subexponential growth condition $(SG_\beta)$.
This condition ensures that the Gevrey-Beurling type space
$$S_q^\infty(G)=\bigcap_{s>0}L^q(G,\omega_s)$$
embeds continuously into the weighted group algebra needed for convolution
estimates. On the operator side, the Dales-Davie algebra associated with the weight
sequence
$M_k={(k!)^{1/\beta}}/{R^k}$
provides the appropriate abstract framework for quantitative inversion;
see \cite[Theorem 2.18]{Grk5}. Combining these estimates yields a
Gevrey version of the noncommutative Wiener lemma for convolution
operators on $L^q(G)$. Related norm-controlled inversion phenomena for weighted convolution
algebras were studied in \cite{Sa1}, where sufficient conditions on
polynomial and certain subexponential weights were obtained to ensure
norm-controlled inversion in the reduced group $C^*$-algebra.
Since  Gevrey-Beurling type space  $S_q^\infty(G)$
is equipped with the projective limit topology induced by the norms
$$\|f\|_{q,s}:=\Big(\int_G|f(x)|^q\omega_s(x)^qd\mu(x)\Big)^{1/q},$$
and $\|f\|_{q,s}$ is monotonically increasing with respect to
$s$, the topology of $S_q^\infty(G)$
is determined by a countable family of norms, making it naturally a Fr\'echet space. We formalize this in the following lemma. As the proof is routine, it is left to the reader.

\begin{lemma}\label{lem1}
Let $\{s_n\}_{n\in\mathbb N}\subset (0,\infty)$ be an increasing sequence such that $s_n\to\infty$, and define
	$$\|f\|_{q,n}:=\left(\int_G |f(x)|^q\omega_{s_n}(x)^qd\mu(x)\right)^{1/q}.$$
Then $$S_q^\infty(G)=\bigcap_{n=1}^\infty L^q(G,\omega_{s_n}).$$
Furthermore, $S_q^\infty(G)$ is a Fr\'echet space when endowed with the locally convex topology generated by the countable family of norms
$\{\|\cdot\|_{q,n}\}_{n\in\mathbb N}$.
\end{lemma}

\medskip

For later use, we introduce the following growth condition.
\begin{defn}
Let $0<\beta<1$. We say that $G$ satisfies the \emph{$\beta$-strong
subexponential growth condition, denoted by $(SG_\beta)$}, if there exists
$0<\gamma\le \beta$ such that for every $c>0$ there is a constant
$C_c>0$ satisfying
$$\mu(B_R)\le C_c e^{cR^\gamma},\qquad R\ge 0.$$
Here $\beta$ is the exponent appearing in the weights
$\omega_s(x)=e^{s\ell(x)^\beta}$.
\end{defn}

We first  establishes that $S_q^\infty(G)$ is continuously embedded into every weighted group  algebra $L^1(G,\omega_s)$.
\begin{lemma}\label{lem2}
Assume that $G$  satisfies condition $(SG_\beta)$ with a proper locally bounded length function $\ell$. Let $1<q<\infty$, $p=\frac{q}{q-1}$, and $s>0$. Then, for every $t>s$, there exists a constant $M_{s,t}>0$ such that
$$\|f\|_{1,s}\le M_{s,t}\|f\|_{q,t} \qquad \text{for all } f\in C_c(G).$$
In particular, $S_q^\infty(G)\subseteq L^1(G,\omega_s)$ for every $s>0$.
\end{lemma}
\begin{proof}
For $s>0$, by Holder's inequality, we have
\begin{align*}
\|f\|_{1,s}=\int_G |f(x)|e^{s\ell(x)^\beta}\,d\mu(x)
&=\int_G \bigl(|f(x)|e^{t\ell(x)^\beta}\bigr)e^{-(t-s)\ell(x)^\beta}d\mu(x)\\	
&\le \|f\|_{q,t}\|e^{-(t-s)\ell^\beta}\|_{L^{p}(G)}.
\end{align*}
It remains to show that
$$e^{-(t-s)\ell^\beta}\in L^p(G).$$
Set $\delta=p(t-s)>0$.
Choose $c>0$ as follows: if $\gamma<\beta$, take any $c>0$, if
$\gamma=\beta$, choose $c$ such that $0<c<\delta$. By the strong
subexponential growth assumption, there exists $C_c>0$ such that
$$\mu(B_R)\le C_c e^{cR^\gamma},\qquad R\ge 0.$$
Let $B_k=\{x\in G:\ell(x)\le k\}$ and
$A_k=B_k\setminus B_{k-1}$ for $k\ge 1$, with $A_0=B_0$. For $x\in A_k$ we
have $\ell(x)> k-1$, hence
\begin{align*}
\|e^{-(t-s)\ell^\beta}\|^{p }_{L^{p}(G)}=\int_G e^{-\delta\ell(x)^\beta}d\mu(x)
&=\sum_{k=0}^\infty \int_{A_k} e^{-\delta\ell(x)^\beta}d\mu(x)\\
&\le \mu(B_0)+\sum_{k=1}^\infty e^{-\delta(k-1)^\beta}\mu(A_k).
\end{align*}
Using $\mu(A_k)\le\mu(B_k)\le C_ce^{ck^\gamma}$, we obtain
\begin{equation}\label{eq41}
\int_Ge^{-\delta\ell(x)^\beta}d\mu(x)
\le \mu(B_0)+C_c\sum_{k=1}^\infty e^{\bigl(-\delta (k-1)^\beta+c k^\gamma\bigr)}.
\end{equation}
If $\gamma<\beta$, then $k^\gamma=o(k^\beta)$ as $k\to\infty$. Hence there exists
$k_0\in\mathbb N$ such that for all $k\ge k_0$,
$$ck^\gamma \le \frac{\delta}{2}(k-1)^\beta.$$
Therefore
$-\delta(k-1)^\beta+ck^\gamma\le-\frac{\delta}{2}(k-1)^\beta$,
and so
$$e^{-\delta(k-1)^\beta+ck^\gamma}\le e^{-\frac{\delta}{2}(k-1)^\beta}\qquad (k\ge k_0).$$
Therefore the series in \ref{eq41} converges.
If $\gamma=\beta$, then $c<\delta=p(t-s)$. Since
$(k-1)^\beta\sim k^\beta$, there exists $\varepsilon>0$ and $k_0\in\mathbb N$ such that
for all $k\ge k_0$,
$$-\delta(k-1)^\beta+ck^\beta\le-\varepsilon k^\beta.$$
Hence the series in \ref{eq41} converges also in this case.
Therefore, set
$$M_{s,t}=\|e^{-(t-s)\ell^\beta}\|_{L^{p}(G)}<\infty,$$
and then
$$\|f\|_{1,s}\le M_{s,t}\|f\|_{q,t}.$$ 
Finally, if $f\in S_q^\infty(G)$, then $\|f\|_{L^q(G,w_t)}<\infty$ for every $t>0$.
Given any $s>0$, choose $t>s$, the above estimate yields
$f\in L^1(G,\omega_s)$. Since $s>0$ was arbitrary, we conclude that
$f\in L^1(G,\omega_s)$ for every $s>0$. This proves the assertion.
\end{proof}

The following lemma shows $S_q^\infty(G)$ is a Fr\'echet $*$-algebra.
\begin{lemma}
Assume that $G$ is unimodular and satisfies condition $(SG_\beta)$ with
respect to a proper locally bounded length function $\ell$. Let
$1<q<\infty$, $p=\frac{q}{q-1}$, and let $s>0$. Then, for every $t>s$,
there exists a constant $C_{s,t}>0$ such that
$$\|f*g\|_{q,s}\le C_{s,t}\|f\|_{q,t}\|g\|_{q,s},\qquad \|f*g\|_{q,s}
\le C_{s,t}\|f\|_{q,s}\|g\|_{q,t}$$
for all $f,g\in S_q^\infty(G)$.
Moreover, $S_q^\infty(G)$ is closed under involution and
$$\|f^*\|_{q,s}=\|f\|_{q,s}.$$
Consequently, endowed with the projective limit topology generated by
$\{\|\cdot\|_{q,n}\}_{n\in\mathbb N}$, $S_q^\infty(G)$ is a Fr\'echet
$*$-algebra.
\end{lemma}

\begin{proof}
Fix $s>0$ and choose any $t>s$.
By Lemma~\ref{lem2}, there exists $M_{s,t}>0$ such that
$$\|h\|_{1,s}\le M_{s,t}\|h\|_{q,t}\qquad \text{for all} \ h\in S_q^\infty(G)).$$
We first prove the two convolution estimates. For $f,g\in S_q^\infty(G)$, let $F_s:=|f|\omega_s$ and $H_s:=|g|\omega_s$. Since $f\in L^1(G,\omega_s)$ and $g\in L^q(G,\omega_s)$, we have
$F_s\in L^1(G)$ and $H_s\in L^q(G)$.  Young's inequality implies that
$F_s*H_s\in L^q(G)$, hence $(F_s*H_s)(x)<\infty$ for almost every $x\in G$. Using $\omega_s(x)\le\omega_s(y)\omega_s(y^{-1}x)$,
we obtain
$$|f(y)||g(y^{-1}x)|\le\omega_s(x)^{-1}F_s(y)H_s(y^{-1}x).$$
Therefore
$$\int_G |f(y)||g(y^{-1}x)|\,dy\le\omega_s(x)^{-1}(F_s*H_s)(x)<\infty.$$
Hence $f*g(x)$ is well defined for almost every $x$, and
$$|(f*g)(x)|\omega_s(x)\le (F_s*H_s)(x).$$
By Young's inequality and Lemma \ref{lem2},
$$\|f*g\|_{q,s}\le\|F_s*H_s\|_q\le\|F_s\|_1\|H_s\|_q=\|f\|_{1,s}\|g\|_{q,s}
\le M_{s,t}\|f\|_{q,t}\|g\|_{q,s}.$$
Similarly,
$$\|f*g\|_{q,s}\le M_{s,t}\|f\|_{q,s}\|g\|_{q,t}.$$
Taking $C_{s,t}=M_{s,t}$ proves the two estimates.  Since $s>0$ was
arbitrary, it follows that $f*g\in S_q^\infty(G)$.
Now we prove stability under involution. Since  $f^*(x)=\overline{f(x^{-1})}$, by unimodularity of $G$ and
$\ell(x^{-1})=\ell(x)$, we have
$$\|f^*\|_{q,s}^q=\int_G |f(x^{-1})|^q\omega_s(x)^qd\mu(x)=\int_G|f(y)|^q \omega_s(y)^qd\mu(y)=\|f\|_{q,s}^q.$$
Thus $S_q^\infty(G)$ is closed under involution.
Finally, by Lemma~\ref{lem1}, $S_q^\infty(G)$ is a Fr\'echet space. To
show that convolution is jointly continuous, it is enough to check the
seminorm estimates. For each $s>0$, choose $t>s$. Since $\omega_s\le
\omega_t$, we have
$$\|h\|_{q,s}\le \|h\|_{q,t}\qquad(h\in S_q^\infty(G)).$$
Therefore, for all $f,g\in S_q^\infty(G)$, 
$$\|f*g\|_{q,s}\le C_{s,t}\|f\|_{q,t}\|g\|_{q,s}
\le C_{s,t}\|f\|_{q,t}\|g\|_{q,t}.$$
This proves joint continuity of convolution. The involution is continuous
because $\|f^*\|_{q,s}=\|f\|_{q,s}$
for every $s>0$. Hence $S_q^\infty(G)$ is a Fr\'echet $*$-algebra.
\end{proof}

\begin{defn}\label{def45}
Assume that  $G$ is a locally compact group with a
proper locally bounded length function $\ell$. Let $\ell$ be the multiplication operator on $L^q(G)$ associated with the length function $\ell$, namely
$$\mathrm{Dom}(D_\ell)=\{\xi\in L^q(G):\ell\xi\in L^q(G)\},
\qquad D_\ell\xi=\ell\xi.$$
We define an operator $\delta_\ell$ on $B(L^q(G))$ as follows. Its domain is
\begin{align*}
\mathrm{Dom}(\delta_\ell)
&=\Big\{T\in B(L^q(G)):T(\mathrm{Dom}(D_\ell))\subseteq \mathrm{Dom}(D_\ell) \text{ and } \\
&[D_\ell,T]\big|_{\mathrm{Dom}(D_\ell)}\text{ extends to a bounded operator on }L^q(G)\Big\}.
\end{align*}
For $T\in\mathrm{Dom}(\delta_\ell)$, we set
$$\delta_\ell(T)=\overline{[D_\ell,T]\big|_{\mathrm{Dom}(D_\ell)}},$$ 
where the bar denotes the unique bounded extension to $L^q(G)$.
\end{defn}

\begin{lemma}\label{closed-derivation}
The operator $\delta_\ell$ defined in Definition~\ref{def45} is a closed
derivation on $B(L^q(G))$. More precisely,
$\operatorname{Dom}(\delta_\ell)$ is a subalgebra of $B(L^q(G))$, and for	all $A,B\in\operatorname{Dom}(\delta_\ell)$,
	$$\delta_\ell(AB)=\delta_\ell(A)B+A\delta_\ell(B).$$
\end{lemma}

\begin{proof}
The multiplication operator $D_\ell$ is closed. Suppose that
$T_n\in\operatorname{Dom}(\delta_\ell)$, $T_n\to T$, and
$\delta_\ell(T_n)\to R$ in $B(L^q(G))$. For every
$\xi\in\operatorname{Dom}(D_\ell)$, we have
$$D_\ell T_n\xi=T_nD_\ell\xi+\delta_\ell(T_n)\xi.$$
The right-hand side converges in $L^q(G)$ to
$TD_\ell\xi+R\xi$, while $T_n\xi\to T\xi$.
Since $D_\ell$ is closed, we get
$T\xi\in\operatorname{Dom}(D_\ell)$ and
$$D_\ell T\xi=TD_\ell\xi+R\xi.$$
Thus, for $\xi\in\operatorname{Dom}(D_\ell)$, we have 
$[D_\ell,T]\xi=R\xi$.
Hence $T\in\operatorname{Dom}(\delta_\ell)$ and $\delta_\ell(T)=R$.
Therefore $\delta_\ell$ is closed. Now let $A,B\in\operatorname{Dom}(\delta_\ell)$. Since both $A$ and $B$ preserve $\operatorname{Dom}(D_\ell)$, we have
$$AB(\operatorname{Dom}(D_\ell))\subseteq\operatorname{Dom}(D_\ell).$$
Moreover, act on $\operatorname{Dom}(D_\ell)$, we have
$$[D_\ell,AB]=[D_\ell,A]B+A[D_\ell,B].$$
Since $[D_\ell,A]\big|_{\operatorname{Dom}(D_\ell)}$ and
$[D_\ell,B]\big|_{\operatorname{Dom}(D_\ell)}$ extend to the bounded
operators $\delta_\ell(A)$ and $\delta_\ell(B)$, respectively, the
operator $[D_\ell,AB]\big|_{\operatorname{Dom}(D_\ell)}$
extends to
$$\delta_\ell(A)B+A\delta_\ell(B)\in B(L^q(G)).$$
Therefore $AB\in\operatorname{Dom}(\delta_\ell)$ and
$$\delta_\ell(AB)=\delta_\ell(A)B+A\delta_\ell(B).$$
Thus $\delta_\ell$ is a derivation. In addition, for this closed derivation, the iterated domains are defined by 
$\operatorname{Dom}(\delta_\ell^0)=B(L^q(G))$,
and, for $k\ge 1$,
$$\operatorname{Dom}(\delta_\ell^k)
=\left\{T\in \operatorname{Dom}(\delta_\ell^{k-1}):
\delta_\ell^{k-1}(T)\in \operatorname{Dom}(\delta_\ell)
\right\}.$$
\end{proof}

Before proceeding to the main norm estimates, it is necessary to verify that the convolution operators with the kernel in  $S_q^\infty(G)$ possess smoothness with respect to the derivation
$\delta_\ell$ in Definition \ref{def45}. The following lemma establishes that these operators belong
to the domain of every iterate of $\delta_\ell$,  ensuring that all iterated commutators are well-defined and extend to bounded operators on $L^q(G)$.
.

\begin{lemma}\label{chensmooth}
Assume that  $G$ is unimodular and satisfies condition $(SG_\beta)$ with respect to a
proper locally bounded length function $\ell$. Let $\delta_\ell$ be the derivation defined in Definition \ref{def45}. Then
$$\lambda_q(S_q^\infty(G))\subseteq
\bigcap_{k=0}^{\infty}\operatorname{Dom}(\delta_\ell^k).$$
More precisely, for every $f\in S_q^\infty(G)$ and every
$k\in\mathbb N$, the operator $\lambda_q(f)$ belongs to
$\operatorname{Dom}(\delta_\ell^k)$, and $\delta_\ell^k(\lambda_q(f))$
extends to a bounded operator on $L^q(G)$.
\end{lemma}

\begin{proof}
For $f\in S_q^\infty(G)$. Let $T_f=\lambda_q(f)$ be  the left convolution operator acting on $L^q(G)$ and 
$\mathcal{D}(D_\ell)=\{\xi\in L^q(G):\ell\xi\in L^q(G)\}$
be the domain of $D_\ell$. We first establish an explicit integral formula for the action of the derivation. Let $\xi\in C_c(G)$. The first commutator is given by
\begin{align*}
(\delta_\ell(T_f)\xi)(x)
&=(D_\ell T_f\xi)(x)-(T_fD_\ell\xi)(x)\\
&=\ell(x\int_G f(y)\xi(y^{-1}x)d\mu(y)-\int_Gf(y)\ell(y^{-1}x)\xi(y^{-1}x)d\mu(y) \\
&=\int_Gf(y)\left(\ell(x)-\ell(y^{-1}x)\right)\xi(y^{-1}x)d\mu(y).
\end{align*}
For $k\in\mathbb N$, define initially on $C_c(G)$ the operator $A_k$ by
\begin{equation}\label{chen41}
(A_k\xi)(x)=\int_G f(y)\left(\ell(x)-\ell(y^{-1}x)\right)^k\xi(y^{-1}x)dy.
\end{equation}
Thus $A_0=T_f$. We claim that $A_k$ extends to a bounded operator on
$L^q(G)$ and that
$$\delta_\ell(A_k)=A_{k+1}.$$
Since $\ell(x)=\ell(y\cdot y^{-1}x)\le\ell(y)+\ell(y^{-1}x)$, we have 
$\ell(x)-\ell(y^{-1}x)\le \ell(y)$. Similarly, 
$\ell(y^{-1}x)\le\ell(y^{-1})+\ell(x)=\ell(y)+\ell(x)$ 
implies $\ell(y^{-1}x)-\ell(x)\le\ell(y)$. Thus,
$$\left|\ell(x)-\ell(y^{-1}x)\right|\le\ell(y).$$
By Eq.\eqref{chen41},  for $\xi\in C_c(G)$, we have
\begin{align*}
|(A_k\xi)(x)| 
&\le \int_G |f(y)| \left|\ell(x)-\ell(y^{-1}x)\right|^k|\xi(y^{-1}x)|d\mu(y)\\
&\le\int_G\left(|f(y)|\ell(y)^k \right)|\xi(y^{-1}x)|d\mu(y)\\	&=\left((|f|\ell^k)*|\xi|\right)(x).
\end{align*}
Therefore, by Young's inequality,
$$\|A_k\xi\|_{L^q(G)}\le\|f\ell^k\|_{L^1(G)} \ \|\xi\|_{L^q(G)}.$$
Since $f\in S_q^\infty(G)$, Lemma~\ref{lem2} implies
$f\in L^1(G,\omega_s)$ for every $s>0$, where $\omega_s(y)=e^{s\ell(y)^\beta}$.
In particular, choosing $s=1$,
since $0<\beta<1$, the subexponential weight
$\omega_1(y)=e^{\ell(y)^\beta}$ dominates any fixed polynomial power $\ell(y)^k$. Specifically, for fixed $k\in \mathbb N$, the function 
$r\rightarrow\frac{r^k}{e^{r^\beta }}(r\geq 0)$ is bounded, because 
$\lim_{r\to\infty}\frac{r^k}{ e^{r^\beta}}=0$ and it is continuous on $[0, +\infty)$. Hence, there exists $C_k=\sup_{r\geq 0}\frac{r^k}{e^{r^\beta}}<\infty$  such that for all $y\in G$
$$\ell(y)^k\le C_k e^{\ell(y)^\beta}=C_k \omega_1(y).$$
Hence
\begin{align*}
\|f\ell^k\|_{L^1(G)}=\int_G|f(y)|\ell(y)^kd\mu(y)
&\le C_k\int_G|f(y)|\omega_1(y)d\mu(y)\\
&=C_k\|f\|_{ L^1(G, \omega_1)}<\infty.
\end{align*}
Thus $A_k$ extends uniquely to a bounded operator on $L^q(G)$ and
\begin{equation}\label{important}
\|A_k\|_{B(L^q(G))}\le\|f\ell^k\|_{L^1(G)}.
\end{equation}
Next let $\xi\in\operatorname{Dom}(D_\ell)$. Using
$\ell(x)\le\ell(y)+\ell(y^{-1}x)$
and $|\ell(x)-\ell(y^{-1}x)|\le\ell(y),$
we obtain
\begin{align*}
\ell(x)|A_k\xi(x)|	
&\le \int_G |f(y)|\ell(x)|\ell(x)-\ell(y^{-1}x)|^k\xi(y^{-1}x)|d\mu(y)\\
&\le\int_G |f(y)|\ell(y)^{k+1}|\xi(y^{-1}x)|d\mu(y)\\
	&+\int_G |f(y)|\ell(y)^k\ell(y^{-1}x)|\xi(y^{-1}x)|d\mu(y)\\
&= \bigl((|f|\ell^{k+1})*|\xi|\bigr)(x)+\bigl((|f|\ell^k)*(\ell|\xi|)\bigr)(x).
\end{align*}
Applying Young's inequality gives
$$\|\ell A_k\xi\|_{L^q(G)}\le \|f\ell^{k+1}\|_{L^1(G)}\|\xi\|_{L^q(G)}+\|f\ell^k\|_{L^1(G)}\|\ell\xi\|_{L^q(G)}<\infty,$$
Therefore
$A_k(\operatorname{Dom}(D_\ell))\subseteq\operatorname{Dom}(D_\ell)$.
Finally, for $\xi\in\operatorname{Dom}(D_\ell)$, a direct computation gives
\begin{align*}
(D_\ell A_k\xi)(x)-(A_kD_\ell\xi)(x)
&=\int_G f(y)\left[\ell(x)-\ell(y^{-1}x)\right]^{k+1}
\xi(y^{-1}x)d\mu(y)\\
&=(A_{k+1}\xi)(x).	
\end{align*}
Since $A_{k+1}$ is bounded on $L^q(G)$, by the definition of
$\delta_\ell$, we get 
$A_k\in\operatorname{Dom}(\delta_\ell)$
and $\delta_\ell(A_k)=A_{k+1}$.
Starting from $A_0=T_f$, by induction we have 
$$T_f\in\operatorname{Dom}(\delta_\ell^k)\qquad\text{and}\qquad\delta_\ell^k(T_f)=A_k$$
for every $k\in\mathbb N_0$. This proves the assertion.
\end{proof}

Having established in Lemma \ref{chensmooth} that the convolution operators are infinitely differentiable with respect to
$\delta_\ell$, we now turn to the quantitative control of these iterated derivations.

\begin{proposition}\label{prochen1}
Assume that  $G$ is unimodular and satisfies condition $(SG_\beta)$ with a locally bounded length function $\ell$. Let $\delta_\ell$ be the derivation defined in Definition \ref{def45} and let  $1<q<\infty$, $p=\frac{q}{q-1}$, and $s>0$.
Then, for every $t>s$,
there exists a constant $M_{s,t}>0$ such that 
$$\sup_{k\ge0}\frac{R_s^k}{(k!)^{1/\beta}}\bigl\|\delta_\ell^k(\lambda_q(f))\bigr\|_{B(L^q(G))}\le M_{s,t}\|f\|_{q,t}$$
for all $f\in S_q^\infty(G)$, where $R_s=(\beta s)^{1/\beta}>0$. 
\end{proposition}

\begin{proof}
For $s>0$ and $f\in S_q^\infty(G)$. By Lemma~\ref{lem2},
there exists $M_{s,t}<\infty$ such that
\begin{equation}\label{eq:le1-used}
		\|f\|_{1,s}\le M_{s,t}\,\|f\|_{q,t}.
\end{equation}
By \eqref{important} in the proof of Lemma \ref{chensmooth}, for every $k\ge 0$, we have
\begin{equation}\label{eq:young-comm}
\bigl\|\delta_\ell^k(T_f)\bigr\|_{B(L^q(G))}\le\|f\ell^k\|_{L^1(G)}.
\end{equation}
For $u\ge 0$ consider $h(u)=u^ke^{-s u^\beta}$. Maximizing $h$ gives
$$\sup_{u>0} u^k e^{-su^\beta}=\Big(\frac{k}{s\beta}\Big)^{k/\beta}e^{-k/\beta}.$$
Since  $(k/e)^k\leq k!$,  we have
$\Big(\frac{k}{e}\Big)^{k/\beta}\le(k!)^{1/\beta}.$
Hence for all $u\ge 0$,
\begin{equation}\label{eq:key-pointwise}
u^k \le s^{-k/\beta}(\beta^{-1/\beta})^k\,(k!)^{1/\beta}e^{s u^\beta}.
\end{equation}
With $u=\ell(x)$ this becomes
$$\ell(x)^k\le(\beta s)^{-k/\beta}(k!)^{1/\beta}\omega_s(x).$$
Therefore
\begin{equation}\label{eq:flk}
\|f\,\ell^k\|_{L^1(G)}
\le (\beta s)^{-k/\beta}(k!)^{1/\beta}\|f\|_{1,s}.
\end{equation} 
Combining\eqref{eq:le1-used} \eqref{eq:young-comm} and \eqref{eq:flk},  for all $k\ge 0$, we get
$$\bigl\|\delta_\ell^k(\lambda_q(f))\bigr\|_{B(L^q(G))}\le M_{s,t}(\beta s)^{-k/\beta}(k!)^{1/\beta}\|f\|_{q,t}.$$
Set $R_s =(\beta s)^{1/\beta}.$
Then
$$\frac{R_s^k}{(k!)^{1/\beta}}\bigl\|\delta_\ell^k(\lambda_q(f))\bigr\|_{B(L^q(G))}
\le M_{s,t}\,\|f\|_{q,t}\quad\text{for all} \ k\ge 0.$$
Taking the supremum over $k$ yields
$$\sup_{k\ge 0}\frac{R_s^k}{(k!)^{1/\beta}}
\bigl\|\delta_\ell^k(\lambda_q(f))\bigr\|_{B(L^q(G))}\le M_{s,t}\|f\|_{q,t}.$$
\end{proof}

The following proposition is the abstract Gevrey inversion principle needed in the proof of the main theorem.  It is a direct consequence of the norm-controlled inversion theory of Gr\"ochenig and Klotz \cite{Grk1,Gr2}, we included the proof here for the completeness. 

\begin{proposition}\label{prochen2}
Let $\delta$ be a closed derivation on a unital Banach algebra $\mathcal B$.  For $\sigma>1$ and $R>0$, set
$$G_R^{(\sigma)}(\mathcal B)=\left\{T\in\bigcap_{n\ge0}\mathrm{Dom}(\delta^n):
\|T\|_{\sigma,R}:=\sup_{n\ge0}\frac{R^n\|\delta^n(T)\|_{\mathcal B}}{(n!)^\sigma}<\infty\right\}.$$
Let $0<R^\prime<R$.  If $T\in G_R^{(\sigma)}(\mathcal B)$ is invertible in $\mathcal B$, then $T^{-1}\in G_{R^\prime}^{(\sigma)}(\mathcal B)$ and
$$\|T^{-1}\|_{\sigma,R^\prime}\le\|T^{-1}\|_{\mathcal B}
 v_{\sigma-1}\!\left(\|T^{-1}\|_{\mathcal B}\frac{\|T\|_{\sigma,R}}{1-R^\prime/R}\right),$$
where
$v_{\sigma-1}(x)=\sum_{m=0}^\infty\frac{x^m}{(m!)^{\sigma-1}}$.
Consequently the operator-algebraic Gevrey-Beurling class 
$$G^{(\sigma)}(\mathcal B)=\bigcap_{R>0}G_R^{(\sigma)}(\mathcal B)$$ 
is inverse-closed in $\mathcal B$.
\end{proposition}

\begin{proof}
Let $T\in G_{R}^{(\sigma)}(\mathcal B)$. By definition, for all $n\ge0$ we have 
$$\frac{R^n\|\delta^n(T)\|_{\mathcal B}}{(n!)^\sigma} \le \|T\|_{\sigma,R}.$$
Recall that Dales-Davie algebra is defined by 
$$\mathcal{D}^1_{M_{R^\prime}}=\left\{T\in{\mathcal B}^\infty:  	\|T\|_{\mathcal{D}^1_{M_{R^\prime}}}=\sum_{n=0}^\infty\frac{{R^\prime} ^n}{(n!)^\sigma} \|\delta^n(T)\|_{\mathcal B}<\infty\right\}.$$
For $R^\prime<R$, since 
\begin{align*}
\|T\|_{\mathcal{D}^1_{M_{R^\prime}}}
&=\sum_{n=0}^\infty \left(\frac{R^\prime}{R}\right)^n \frac{R^n\|\delta^n(T)\|_{\mathcal B}}{(n!)^\sigma}\\
&\le\sum_{n=0}^\infty\left(\frac{R^\prime}{R}\right)^n\|T\|_{\sigma, R}
=\frac{1}{1-R^\prime/R}\|T\|_{\sigma, R}<\infty,
\end{align*}
we have $T\in\mathcal{D}^1_{M_{R^\prime}}$.  Theorem 2.18 in \cite{Grk5} established that, for a sequence $M_k$, if its combinatorial coefficients satisfy 
$$A_m=\big(\{\sup\{\frac{k!}{M_k}\prod_{j=1}^{m}\frac{M_{l_j}}{l_j!}: l_j\geq 1 \ \text{for} \ 1\leq j\leq m, \ \sum_{j=1}^{m}l_j=k\}\big)^{\frac{1}{m}} \to 0,$$ 
then Dales-Davie algebra admits norm-controlled inversion. In our setting, the sequence is $M_k=\frac{(k!)^\sigma}{(R^\prime)^k}$. When calculating the coefficient $A_m$, the  $R^\prime$ perfectly cancels and the supremum reduces 
\begin{equation}\label{eqchen65}
A_m=\left(\sup_{l_j \ge 1}\left(\frac{l_1!\cdots l_m!}{k!}\right)^{\sigma-1}\right)^{1/m}= (m!)^{\frac{1-\sigma}{m}},
\end{equation}
(see \cite[Eq.(2.25), Page 928 ]{Grk5}.
Since $\sigma>1$, we have $A_m \to 0$. Therefore,   according to the  \cite[Eq (2.24), page 927]{Grk5} , we have
$$\|T^{-1}\|_{\mathcal{D}^1_{M_{R^\prime}}}\le\sum_{m=0}^\infty A_m^m\|T^{-1}\|_{\mathcal B }^{m+1}\|T\|_{\mathcal{D}^1_{M_{R^\prime}}}^m.$$
By Eq. \eqref{eqchen65}, we obtain
\begin{align}\label{eqchen66}
\|T^{-1}\|_{\mathcal{D}^1_{M_{R^\prime}}} 
&\le\|T^{-1}\|_{\mathcal B}\sum_{m=0}^\infty\frac{1}{(m!)^{\sigma-1}}\left( \|T^{-1}\|_{\mathcal B}\|T\|_{\mathcal{D}^1_{M_{R^\prime}}}\right)^m \\ \nonumber  
&=\|T^{-1}\|_{\mathcal B}\cdot v_{\sigma-1} \left(\|T^{-1}\|_{\mathcal B} \|T\|_{\mathcal{D}^1_{M_{R^\prime}}}\right)
\end{align}
Substituting the bound 
$\|T\|_{\mathcal{D}^1_{M_{R^\prime}}}\le\frac{\|T\|_{\sigma,R}}{1-R^\prime/R}$,  
we get
$$\|T^{-1}\|_{\mathcal{D}^1_{M_{R^\prime}}}\le\|T^{-1}\|_{\mathcal B}\cdot v_{\sigma-1} \left(\|T^{-1}\|_{\mathcal B}\frac{\|T\|_{\sigma, R}}{1-R^\prime/R}\right),$$
where 
$v_{\sigma-1}(x)=\sum_{m=0}^\infty\frac{x^m}{(m!)^{\sigma-1}}$.
Finally,  since the  supremum of a sequence is bounded by its $\ell^1$-sum,  we have	
\begin{align*}	
\|T^{-1}\|_{\sigma, R^\prime} 
&= \sup_{n \ge 0}\frac{(R^\prime)^n\|\delta^n(T^{-1})\|_{\mathcal B}}{(n!)^\sigma} \\
&\le\sum_{n=0}^\infty\frac{(R^\prime)^n\|\delta^n(T^{-1})\|_{\mathcal B}}{(n!)^\sigma}= \|T^{-1}\|_{\mathcal{D}^1_{M_{R^\prime}}}
\end{align*}
By \eqref{eqchen66}, we  obtain
$$\|T^{-1}\|_{\sigma, R^\prime}\le\|T^{-1}\|_{\mathcal B}\cdot v_{\sigma-1}\left( \|T^{-1}\|_{\mathcal B}\frac{\|T\|_{\sigma, R}}{1-R^\prime/R}\right).$$
It remains to prove that the Gevrey--Beurling class is inverse-closed. Let
$T\in G^{(\sigma)}(\mathcal B)=\bigcap_{R>0}G_R^{(\sigma)}(\mathcal B)$
be invertible in $\mathcal B$. Fix $R^\prime>0$. Since
$T\in G_{2R^\prime}^{(\sigma)}(\mathcal B)$, the preceding estimate with
$R=2R^\prime$ gives
$T^{-1}\in G_{R^\prime}^{(\sigma)}(\mathcal B)$.
As $R^\prime>0$ was arbitrary, we have
$$T^{-1}\in\bigcap_{R^\prime>0}G_{R^\prime}^{(\sigma)}(\mathcal B)=G^{(\sigma)}(\mathcal B).$$
Thus $G^{(\sigma)}(\mathcal B)$ is inverse-closed in
$\mathcal B$. This completes the proof.
\end{proof}

We now apply the abstract inverse-closedness criterion of
Proposition~\ref{prochen2} to the derivation $\delta_\ell$ on
$B(L^q(G))$, restricted to the intersection of the unitized algebra
$\widetilde{PF_q(G)}$ with the operator-algebraic Gevrey-Beurling class. This yields the
following inverse-closeness theorem, which may be viewed as a Gevrey-type noncommutative
Wiener lemma. It is worth emphasizing that this theorem does not require 
$G$ to satisfy condition $(SG_\beta)$.

\begin{theorem}\label{inverse-closed}
Assume that  $G$ is a locally compact group with a proper locally bounded length function $\ell$. Let $\delta_\ell$ be the derivation defined in Definition \ref{def45}. For 
$1<q<\infty$ and $0<\beta<1$, set 
$$\widetilde{PF_q(G)}=PF_q(G)+\mathbb C I\subset B(L^q(G))$$
be the unitization of $PF_q(G)$  and  
$$G^{(1/\beta)}(B(L^q(G)))=\bigcap_{R>0}G_{R}^{(1/\beta)}(B(L^q(G))).$$
be the be operator-algebraic Gevrey-Beurling class. 
Then the algebra
$$\mathcal A_{q,\beta}(G)=\widetilde{PF_q(G)}\cap G^{(1/\beta)}(B(L^q(G)))$$
is inverse-closed in $\widetilde{PF_q(G)}$. 
Moreover, the Gevrey seminorms of $A^{-1}$ satisfy the estimates in Proposition~\ref{prochen2}.
\end{theorem}

\begin{proof}
By Lemma \ref{closed-derivation}, $\delta_\ell$ is a closed derivation. Since $\sigma>1$, by Remark \ref{rek24}, $G^{(\sigma)}(B(L^q(G)))$ is an Fr\'echet algebra.  
Hence $\mathcal A_{q,\beta}(G)$ is an algebra as the intersection of two subalgebras of $B(L^q(G))$. Now let $A\in \mathcal A_{q,\beta}(G)$ and suppose that $A$ is invertible
in $\widetilde{PF_q(G)}$.  Then there exists
$B\in \widetilde{PF_q(G)}$ such that
$AB=BA=I$.	Since $\widetilde{PF_q(G)}\subset B(L^q(G))$, the same identities hold in $B(L^q(G))$.  Hence $A$ is invertible in $B(L^q(G))$, and its inverse in $B(L^q(G))$ is precisely $B$.  In particular,
$A^{-1}=B\in\widetilde{PF_q(G)}$. On the other hand, since 
$A\in G^{(\sigma)}(B(L^q(G)))$,
by Proposition \ref{prochen2}, applied with
$\mathcal B=B(L^q(G))$ and $\delta=\delta_\ell$, the Gevrey-Beurling
class $G_\ell^{(\sigma)}(B(L^q(G)))$ is inverse-closed in $B(L^q(G))$. Hence
$$A^{-1}\in G^{(\sigma)}(B(L^q(G))).$$
Moreover, the seminorms of $A^{-1}$ satisfy the estimate stated in
Proposition \ref{prochen2}.
Therefore
$$A^{-1}\in\widetilde{PF_q(G)}\cap G^{(\sigma)}(B(L^q(G)))=\mathcal A_{q,\beta}(G).$$
This shows $\mathcal A_{q,\beta}(G)$ is inverse-closed in
$\widetilde{PF_q(G)}$.
\end{proof}

\begin{lemma}\label{density}
Assume that  $G$ is unimodular and satisfies condition $(SG_\beta)$ with a locally bounded length function $\ell$. The algebra $\mathcal A_{q,\beta}(G)$ is dense in
$\widetilde{PF_q(G)}$.
\end{lemma}

\begin{proof}
By Proposition~\ref{prochen1}, for every $s>0$ there exist $t>s$ and
$M_{s,t}<\infty$ such that, for all $f\in S_q^\infty(G)$,
$$\sup_{k\ge0}\frac{R_s^k}{(k!)^{1/\beta}}
\left\|\delta_\ell^k(\lambda_q(f))\right\|_{B(L^q(G))}\le M_{s,t}\|f\|_{q,t},
\qquad R_s=(\beta s)^{1/\beta}.$$
Now let $R>0$ be arbitrary and choose
$s=\frac{R^\beta}{\beta}$.
Then $s>0$ and $R_s=R$. Hence, for this choice of $s$, Proposition~\ref{prochen1}
gives
$$\sup_{k\ge0}\frac{R^k}{(k!)^{1/\beta}}\left\|\delta_\ell^k(\lambda_q(f))\right\|_{B(L^q(G))}\le M_{s,t}\|f\|_{q,t}<\infty.$$
Since $R>0$ was arbitrary, it follows that
$$\lambda_q(f)\in G^{(1/\beta)}(B(L^q(G))).$$
Since $\lambda_q(f)\in PF_q(G)\subset\widetilde{PF_q(G)}$, it follows that
$\lambda_q(S_q^\infty(G))\subset\mathcal A_{q,\beta}(G)$.
Moreover, $\lambda_q(S_q^\infty(G))$ is dense in $PF_q(G)$. Therefore,
for every $T\in PF_q(G)$, there exists a sequence
$f_n\in S_q^\infty(G)$ such that $\lambda_q(f_n)\to T$ in $PF_q(G)$.
Since $\delta_\ell(I)=0$, we have $I\in G^{(1/\beta)}(B(L^q(G)))$.
Thus $I\in\mathcal A_{q,\beta}(G)$.
Now let
$$A=\lambda I+T\in\widetilde{PF_q(G)}$$
for $\lambda\in\mathbb C$ and $T\in PF_q(G)$. Then 
$\lambda I+\lambda_q(f_n)\in\mathcal A_{q,\beta}(G)$
and
$$\lambda I+\lambda_q(f_n)\rightarrow \lambda I+T=A$$
in $\widetilde{PF_q(G)}$. Hence $\mathcal A_{q,\beta}(G)$ is dense in
$\widetilde{PF_q(G)}$.
\end{proof}

\begin{corollary}\label{K}
Assume that $G$ is unimodular and satisfies condition $(SG_\beta)$ with a proper locally bounded length function $\ell$. Let
$$\mathcal A_{q,\beta}(G)=\widetilde{PF_q(G)}\cap G^{(1/\beta)}\bigl(B(L^q(G))\bigr),$$
equipped with the natural intersection topology. Then
$\mathcal A_{q,\beta}(G)$ is a dense Fréchet subalgebra of
$\widetilde{PF_q(G)}$ stable under holomorphic functional calculus.
Consequently, the inclusion
$\mathcal A_{q,\beta}(G)\hookrightarrow \widetilde{PF_q(G)}$
induces an isomorphism in topological $K$-theory,
$$K_*\bigl(\mathcal A_{q,\beta}(G)\bigr)\cong
K_*\bigl(\widetilde{PF_q(G)}\bigr).$$
\end{corollary}

\begin{proof}
Put $\mathcal P=\widetilde{PF_q(G)}$ and
$\mathcal G=G^ {(1/\beta)}(B(L^q(G)))$.
We equip
$$\mathcal A_{q,\beta}(G)=\mathcal P\cap\mathcal G$$
with the natural intersection topology, equivalently the topology generated
by the seminorms
$$\rho_n(A)=\|A\|_{\mathcal P}+p_n(A),\qquad n\in\mathbb N,$$
where
$$p_n(A)=\sup_{k\ge0}\frac{n^k}{(k!)^{1/\beta}}
	\|\delta_\ell^k(A)\|_{B(L^q(G))}.$$
Since $\mathcal P$ is a unital Banach algebra and $\mathcal G$ is a
unital Fréchet algebra, both continuously embedded in $B(L^q(G))$, their
intersection $\mathcal A_{q,\beta}(G)$, endowed with the natural
intersection topology, is a Fréchet algebra. Moreover,
$I\in\mathcal A_{q,\beta}(G)$, because $I\in\mathcal P$ 
and $\delta_\ell(I)=0$. Hence $\mathcal A_{q,\beta}(G)$ is unital.
By Lemma \ref{density}, $\mathcal A_{q,\beta}(G)$ is dense in
$\widetilde{PF_q(G)}$. By Theorem~\ref{inverse-closed},
$\mathcal A_{q,\beta}(G)$ is inverse-closed in $\widetilde{PF_q(G)}$.
Since the two algebras have the same unit, inverse-closedness is
equivalent to spectral invariance in $\widetilde{PF_q(G)}$.
As $\mathcal P$ is a Banach algebra, it is an $m$-convex Fréchet
$Q$-algebra.  Hence, by \cite[Lemma~1.2]{Schweitzer}, spectral invariance
is equivalent to stability under holomorphic functional calculus in
$\mathcal P$.  Therefore $\mathcal A_{q,\beta}(G)$ is a dense
holomorphically closed Fréchet subalgebra of $\mathcal P$.  The
$K$-theoretic invariance of dense local Fréchet subalgebras then gives
$$K_*\bigl(\mathcal A_{q,\beta}(G)\bigr)
\cong K_*\bigl(\mathcal P\bigr)=K_*\bigl(\widetilde{PF_q(G)}\bigr),$$
see \cite[Proposition~8.14]{Va}.
\end{proof}

\begin{remark}
In noncommutative geometry, cyclic cocycles and unbounded geometric
cycles are usually defined on dense smooth subalgebras rather than on
the ordinary Banach or $C^*$-algebra itself.  Such subalgebras provide the
regularity needed to control commutators, higher derivations, and
kernel estimates.  When the smooth subalgebra is stable under
holomorphic functional calculus, it has the same topological $K$-theory as
the ordinary algebra; hence cyclic cocycles defined on it pair
canonically with $K$-theory classes of the ordinary algebra.
Corollary~\ref{K} identifies $\mathcal A_{q,\beta}(G)$
as a holomorphically closed Fréchet subalgebra of $\widetilde{PF_q(G)}$.  
This is useful because the smoothness here is produced by
subexponential Gevrey estimates on iterated commutators with the length
function, rather than by polynomial rapid decay.  Thus the construction
fits naturally into the framework of analytic and local cyclic homology
and suggests applications to higher index theory
\cite{PiazzaSchick2014,XieYu2017,Yu2020, Puschnigg2023}, especially for groups of intermediate or strong subexponential growth where classical rapid decay smooth algebras may be unavailable.
\end{remark}

The main result of this section is the following application of the
Theorem \ref{inverse-closed} to convolution operators with kernel in
$S_q^\infty(G)$. It yields quantitative Gevrey regularity estimates for
their inverses.

\begin{theorem}\label{main2}
Assume that $G$ is unimodular and satisfies condition $(SG_\beta)$ with
respect to a proper locally bounded length function $\ell$, where
$0<\beta<1$. Let $1<q<\infty$. For $f\in S_q^\infty(G)$, let
$T_f:=\lambda_q(f)$ be the associated convolution operator on $L^q(G)$.
If $T_f$ is invertible in the unitization $\widetilde{PF_q(G)}$, then
$$T_f^{-1}\in\mathcal A_{q,\beta}(G)
	=\widetilde{PF_q(G)}\cap G^{(1/\beta)}\bigl(B(L^q(G))\bigr),$$
where the Gevrey class is taken with respect to the derivation
$\delta_\ell$. More precisely, for every $R>0$ and every $0<R^\prime<R$, we have
$$\|T_f^{-1}\|_{1/\beta,R^\prime}\le\|T_f^{-1}\|_{B(L^q(G))}
v_{1/\beta-1}\left(\|T_f^{-1}\|_{B(L^q(G))}
	\frac{\|T_f\|_{1/\beta,R}}{1-R^\prime/R}\right),$$
where
$$v_{1/\beta-1}(x)=	\sum_{m=0}^{\infty}\frac{x^m}{(m!)^{1/\beta-1}}.$$
\end{theorem}
\begin{proof}
By Proposition~\ref{prochen1}, for every $R>0$ the operator
$T_f=\lambda_q(f)$ belongs to $G_R^{(1/\beta)}(B(L^q(G)))$. Hence
	$T_f\in G^{(1/\beta)}(B(L^q(G)))$.
Moreover, by definition of $PF_q(G)$, we have
$T_f\in PF_q(G)\subset \widetilde{PF_q(G)}$.
Therefore $T_f\in\mathcal A_{q,\beta}(G)$.
Assume that $T_f$ is invertible in $\widetilde{PF_q(G)}$. By
Theorem \ref{inverse-closed}, the algebra
$\mathcal A_{q,\beta}(G)$ is inverse-closed in
$\widetilde{PF_q(G)}$. Hence $T_f^{-1}\in\mathcal A_{q,\beta}(G)$.
Finally, the displayed estimate follows directly from
Proposition \ref{prochen2}, applied with
$$\mathcal B=B(L^q(G)),\qquad\delta=\delta_\ell,\qquad\sigma=1/\beta.$$
\end{proof}

\section{Examples}
In this section we give examples illustrating the range of applicability of
the Gevrey inverse-closedness theorem.  These examples show that the theorem
naturally applies to groups with strong subexponential growth, beyond the
usual polynomial rapid decay setting.

\begin{example}\label{ex1}
Every compactly generated locally compact group of polynomial growth
satisfies the strong subexponential growth condition for every
$0<\gamma<1$. In particular, this applies to connected nilpotent Lie
groups, such as the Heisenberg group, which have polynomial growth by
Guivarc'h's theorem \cite{Guivarch1973}.
	
Among compactly generated connected locally compact groups, there are no
groups of intermediate growth: such groups have either polynomial growth
or exponential growth \cite{JENK}. Hence, in the connected setting, our
theorem applies to the polynomial-growth side of this dichotomy. The
intermediate-growth examples relevant to the present paper arise instead
in the discrete, or more generally totally disconnected, setting.
\end{example}

\begin{example}\label{ex7}
The Grigorchuk group was introduced by Grigorchuk in \cite{Grig} as the
first example of a finitely generated group of intermediate growth.  We take $\ell$ to be the word length associated with a finite generating set.  Its growth function
	$$\nu(n)=\#\{x:\ell(x)\le n\}$$
satisfies the estimate
	$\nu(n)\le C\exp(n^\alpha)$
for some $C>0$ and
$\alpha\approx 0.767<1$
by Bartholdi's upper bound \cite{Barth}.  Hence, given $\beta\in(\alpha,1)$, choose $\gamma $ such that $\alpha<\gamma\le\beta$. Then the Grigorchuk group satisfies condition $(SG_\beta)$. 
Consequently, Corollary \ref{K} and
Theorem~\ref{main2} apply to the Grigorchuk group.
\end{example}

\begin{example}
Let $G$ be a finitely generated group with word length $|\cdot|$, and
suppose that $G$ has strong subexponential growth of exponent
$\alpha<1$. Choose parameters
$\alpha<\beta<\theta<1$ and let $a>0$. Define
$$f(g)=\exp(-a|g|^\theta).$$
Then, for every $s>0$ and every $1<q<\infty$, we have
$$e^{s|\cdot|^\beta}f\in\ell^q(G).$$
Indeed, since $\theta>\beta$, there exist constants $c>0$ and $R_0>0$
such that
$$s|g|^\beta-a|g|^\theta\le-c|g|^\theta$$
whenever $|g|\ge R_0$. Since $G$ has strong subexponential growth of
exponent $\alpha$ and $\theta>\alpha$, the function
$g\mapsto e^{-c|g|^\theta}$ belongs to $\ell^q(G)$. Hence
$f\in S_q^\infty(G)$ with respect to the subexponential weights of order
$\beta$. Consequently, whenever $T_f=\lambda_q(f)$ is invertible in
$\widetilde{PF_q(G)}$, Theorem~\ref{main2} gives
$T_f^{-1}\in\mathcal A_{q,\beta}(G)$.
In particular, this applies to the explicit invertible family
$$
T=I+\lambda_q(\varepsilon f),
$$
where $\varepsilon>0$ is chosen sufficiently small. Indeed, since
$f\in\ell^1(G)$, Young's inequality gives
$$\|\lambda_q(\varepsilon f)\|_{B(\ell^q(G))}
\le\varepsilon\|f\|_{\ell^1(G)}.$$
Choosing $\varepsilon\|f\|_{\ell^1(G)}<1$, then
$T$ is invertible in $\widetilde{PF_q(G)}$. Therefore,
$T^{-1}\in \mathcal A_{q,\beta}(G)$.
Thus this provides an explicit family of invertible convolution
operators whose inverses have Gevrey regularity of order $1/\beta$ with
respect to the commutator derivation.
\end{example}

\begin{example}\label{ex2}
The strong subexponential growth condition is stable under direct
products. Let $G_1$ and $G_2$ be locally compact groups with locally
bounded length functions $\ell_1$ and $\ell_2$. Assume that, for
$i=1,2$, the group $G_i$ has strong subexponential growth with exponent
$\gamma_i$, where $0<\gamma_i<1$, that is, for every $c>0$ there exists
$C_{i,c}>0$ such that
$$\mu_i(B_i(R))\le C_{i,c}e^{cR^{\gamma_i}}$$
for all $R\ge 0$. Equip $G_1\times G_2$ with the product length
$$\ell(g_1,G_2)=\ell_1(g_1)+\ell_2(G_2).$$
Then $G_1\times G_2$ has strong subexponential growth with exponent
$$\gamma=\max\{\gamma_1,\gamma_2\}.$$ 
Thus Theorems~\ref{main2}  apply to $G_1\times G_2$ whenever 
$\gamma\le\beta < 1$. 
In particular, let $\mathfrak{G}$ be  Grigorchuk group and let $H$ be  any compactly generated locally compact group  of polynomial growth, such as $\mathbb{Z}^d$ or a connected nilpotent Lie group. Since $H$ has strong subexponential growth with every exponent in $(0, 1)$, and $\mathfrak{G}$ has intermediate growth bounded above by $e^{c R^\alpha}$ for some $\alpha \approx 0.767<1$ (see Example \ref{ex7}), so the product $\mathfrak{G} \times H$ has strong subexponential growth with every exponent $\alpha_1\in (\alpha, 1)$. Consequently, Theorem~\ref{main2} applies to $\mathfrak G\times H$ for every $\beta\in(\alpha_1, 1)$.
\end{example}

\begin{example}
Let $K$ be a compact normal subgroup of a locally compact group $G$, and
let $q:G\to G/K$ be the quotient map. Let $\ell_{G/K}$ be a locally bounded length
function on $G/K$, and equip $G$ with the lifted length
$\ell_G=\ell_{G/K}\circ q$.
Then the strong subexponential growth condition is equivalent for
$G$ and $G/K$ with respect to these length functions.
Indeed, the balls of radius $R$ satisfy $$B_G(R)=q^{-1}(B_{G/K}(R)).$$
With compatible choices of Haar measures, Weil's integration formula (see \cite[Corollary B.1.7]{Be}) gives
$$\mu_G(q^{-1}(A))=\mu_K(K)\mu_{G/K}(A)$$
for every Borel set $A\subseteq G/K$. Hence
$$\mu_G(B_G(R))=\mu_K(K)\mu_{G/K}(B_{G/K}(R)).$$
Since $K$ is compact, we have $0<\mu_K(K)<\infty$.
Therefore, for every $0<\beta<1$, the group $G$ satisfies the condition $(SG_\beta)$ with respect to
$\ell_G$ if and only if $G/K$ satisfies the same condition with respect
to $\ell_{G/K}$.
Consequently, if $G$ fits into a short exact sequence
$$1\longrightarrow K\longrightarrow G
\overset{\pi}{\longrightarrow}G_0\longrightarrow 1,$$
where $K$ is compact and $G_0$ is one of the groups from the preceding
examples \ref{ex2}, then $G$ satisfies precisely the same strong subexponential growth conditions as $G_0$ with respect to the lifted length
$\ell_G(g)=\ell_{G_0}(\pi(g))$.
Hence Theorems \ref{main2} apply directly to $G$ whenever
they apply to $G_0$.
\end{example}

\section{Relative Gevrey regularity for pairs of groups}

In this section, motivated by Chatterji and Zarka's relative rapid decay
framework~\cite{ChatterjiZarka2024}, and by the classical rapid decay
theory and its applications
\cite{Haagerup,Jo1,ConnesMoscovici1990,Laff,yu,chat2}, we develop a
relative Gevrey regularity theory for pairs $(G,H)$, where $G$ is a
finitely generated discrete group and $H$ is finitely generated subgroup of $G$.
More precisely, we replace polynomial Sobolev-type estimates with
Gevrey-type factorial bounds for iterated commutators associated with
the quotient length on $G/H$.
Let $\ell$ be the
word length associated with a fixed finite symmetric generating set $S$
of $G$. The quasi-regular representation of $G$ on $\ell^q(G/H)$ is
denoted by $\pi_q^H$ and is given by
$$(\pi_q^H(g)\xi)(x)=\xi(g^{-1}x),\qquad g\in G,\ x\in G/H.$$
For $f\in\mathbb C[G]$, set
$$\pi_q^H(f)=\sum_{g\in G}f(g)\pi_q^H(g).$$
The quotient length on $G/H$ is defined by
$$\mathbb L_H(gH)=\inf_{h\in H}\ell(gh),\qquad g\in G.$$
This definition is independent of the representative of the coset. We
now introduce the relative commutator derivation associated with
$\mathbb L_H$.

\begin{defn}\label{def61}
Let $D_H$ be the multiplication operator on $\ell^q(G/H)$ defined by
$$\operatorname{Dom}(D_H)
=\{\xi\in \ell^q(G/H):\mathbb L_H\xi\in\ell^q(G/H)\},\qquad D_H\xi=\mathbb L_H\xi.$$
We define the operator  $\delta_H$ on
$B(\ell^q(G/H))$ as follows. Its domain is
\begin{align*}
\operatorname{Dom}(\delta_H)
&=\Big\{T\in B(\ell^q(G/H)):
T(\operatorname{Dom}(D_H))\subseteq \operatorname{Dom}(D_H)\text{ and }\\
&[D_H,T]\big|_{\operatorname{Dom}(D_H)}\text{extends to a bounded operator on }\ell^q(G/H)\Big\}.
\end{align*}
For $T\in\operatorname{Dom}(\delta_H)$, set
$$\delta_H(T)=\overline{[D_H,T]\big|_{\operatorname{Dom}(D_H)}},$$
where the bar denotes the unique bounded extension to $\ell^q(G/H)$.
\end{defn}

\begin{lemma}\label{lem:delta-H-closed-derivation}
The operator $\delta_H$ defined above is a closed derivation on
$B(\ell^q(G/H))$. 
\end{lemma}

\begin{proof}
The proof is the same as that of Lemma~\ref{closed-derivation}, replacing $D_\ell$ and $L^q(G)$ by $D_H$ and $\ell^q(G/H)$, respectively.
\end{proof}

\begin{lemma}\label{le63}
Fix $g\in G$, set
$$\varphi_g(x)=\mathbb L_H(x)-\mathbb L_H(g^{-1}x), \qquad x\in G/H.$$
Then $\varphi_g$ is bounded.
\end{lemma}
 
\begin{proof}
Let $aH\in G/H$ and
$\varepsilon>0$. Since $\mathbb L_H(aH)=\inf_{h\in H}\ell(ah)$, we can choose $h_\varepsilon\in H$ such that
$$\ell(ah_\varepsilon)\leq \mathbb L_H(aH)+\varepsilon.$$
Then, by the triangle inequality,
$$\mathbb L_H(g^{-1}aH)\leq\ell(g^{-1}ah_\varepsilon)\leq
\ell(g)+\ell(ah_\varepsilon)\leq\ell(g)+\mathbb L_H(aH)+\varepsilon.$$
Letting $\varepsilon\to 0$, we get
$$\mathbb L_H(g^{-1}aH)-\mathbb L_H(aH)\leq\ell(g).$$
Applying the same argument with $aH$ replaced by $g^{-1}aH$ and $g$
replaced by $g^{-1}$ gives
$$\mathbb L_H(aH)-\mathbb L_H(g^{-1}aH)\leq\ell(g).$$
Therefore
$|\varphi_g(aH)|\leq\ell(g)$.
Since $aH\in G/H$ was arbitrary, it follows that
$$|\varphi_g(x)|\leq\ell(g),\qquad x\in G/H.$$
Hence $\varphi_g$ is bounded on $G/H$, that is 
$$\|\varphi_g\|_\infty=\sup_{x\in G/H}|\varphi_g(x)|\leq\ell(g).$$
\end{proof}

For $g\in G$, define
$$\omega_H(g)=\|\varphi_g\|_\infty
=\sup_{x\in G/H}\left|\mathbb L_H(x)-\mathbb L_H(g^{-1}x)\right|.$$
By Lemma~\ref{le63},  
$\omega_H(g)\leq\ell(g)$ for all $g\in G$. Moreover, $\omega_H$
is symmetric and subadditive. Indeed,
$$\omega_H(g^{-1})=\sup_{x \in G/H}|\mathbb L_H(x)-\mathbb L_H(gx)|= \sup_{y \in G/H}|\mathbb L_H(g^{-1}y)-\mathbb L_H(y)|=\omega_H(g),$$ 
and, for $g_1,g_2\in G$,
\begin{align*}
	\omega_H(g_1g_2) &= \sup_{x \in G/H} |\mathbb L_H(x)-\mathbb L_H(g_2^{-1}g_1^{-1}x)| \\
	&\le \sup_{x \in G/H} |\mathbb L_H(x)-\mathbb L_H(g_1^{-1}x)|+\sup_{x \in G/H} |\mathbb L_H(g_1^{-1}x) - \mathbb L_H(g_2^{-1}(g_1^{-1}x))|\\
	&\le \omega_H(g_1) + \omega_H(g_2),
\end{align*}
Consequently, for $s>0$ and $0<\beta<1$, the function
$\Omega_s(g)=e^{s\omega_H(g)^\beta}$
is submultiplicative, since
$$\omega_H(g_1g_2)^\beta\leq\omega_H(g_1)^\beta+\omega_H(g_2)^\beta.$$
Thus $\Omega_s$ is a subexponential weight of exponent $\beta$.

\begin{defn}
Let $0<\beta<1$. For $s>0$, set
$$\|f\|_{1,s,H}=\sum_{g\in G}|f(g)|e^{s\omega_H(g)^\beta}.$$
\emph{The relative weighted  Gevrey-Beurling type space}  on $G$ is defined by
$$\mathcal S_{1,H}^{\infty}(G)=\bigcap_{s>0}\ell^1\Big(G,e^{s\omega_H^\beta}\Big),$$
endowed with the Fréchet topology generated by the seminorms
$\{\|\cdot\|_{1,s,H}\}_{s>0}$.
Let $1<q<\infty$, $\sigma>1$, and $R>0$. Let $\delta_H$ be the 
closed derivation defined in Definition \ref{def61}. Define
$$G_{H,R}^{(\sigma)}(B(\ell^q(G/H)))
=\left\{T\in\bigcap_{n\ge0}\mathrm{Dom}(\delta_H^n):\|T\|_{\sigma,R,H}<\infty\right\},$$
where
$$\|T\|_{\sigma,R,H}=\sup_{n\ge0}\frac{R^n\|\delta_H^n(T)\|_{B(\ell^q(G/H))}}{(n!)^\sigma}.$$
The \emph {operator-algebraic relative Gevrey-Beurling class} associated with $\delta_H$ is defined by
$$G_H^{(\sigma)}(B(\ell^q(G/H)))
	=\bigcap_{R>0}G_{H,R}^{(\sigma)}(B(\ell^q(G/H))),$$
equipped with the projective limit topology generated by the seminorms
$\{\|\cdot\|_{\sigma,R,H}\}_{R>0}$.
\end{defn}

\begin{remark}
The space $\mathcal S_{1,H}^{\infty}(G)$ is a Fr\'echet 
algebra. Indeed, $\Omega_s(g)=e^{s\omega_H(g)^\beta}$ is
submultiplicative for every $s>0$, and hence $\ell^1(G,\Omega_s)$ is a
weighted Banach algebra with norm $\|\cdot\|_{1,s,H}$. Since
$\{\|\cdot\|_{1,s,H}\}_{s>0}$ is increasing and generated by
$\{\|\cdot\|_{1,n,H}\}_{n\in\mathbb N}$, the space
$\mathcal S_{1,H}^{\infty}(G)$ is a Fr\'echet  algebra as a
projective limit of weighted Banach  algebras.
Since $\delta_H$ is a closed derivation on $B(\ell^q(G/H))$, the same
argument as in Remark~\ref{rek24}, applied to
$\mathcal A=B(\ell^q(G/H))$ and $\delta=\delta_H$, shows that
$G_H^{(\sigma)}(B(\ell^q(G/H)))$
is a Fr\'echet algebra.
\end{remark}

\begin{lemma}\label{relative}
Let $1<q<\infty$ and let $\delta_H$ be the commutator derivation in Definition \ref{def61}. For any $n\geq 0$ and $g\in G$, we have
$\pi_q^H(g)\in\operatorname{Dom}(\delta_H^n)$,
and
$$\|\delta_H^n(\pi_q^H(g))\|_{B(\ell^q(G/H))}\leq\omega_H(g)^n,$$
where
$$\omega_H(g)=\sup_{x\in G/H}|\mathbb L_H(x)-\mathbb L_H(g^{-1}x)|.$$
If $g=e$, we set $\omega_H(g)^0=1.$
Furthermore, if $f\in\ell^1(G)$ satisfies
$$\sum_{g\in G}|f(g)|\omega_H(g)^n<\infty,$$
then $\pi_q^H(f)\in\operatorname{Dom}(\delta_H^n)$
and
$$\|\delta_H^n(\pi_q^H(f))\|_{B(\ell^q(G/H))}\leq
	\sum_{g\in G}|f(g)|\omega_H(g)^n.$$
\end{lemma}
\begin{proof}
Fix $g\in G$, by Lemma \ref{le63}, we have 
$$\|\varphi_g\|_\infty=\sup_{x\in G/H}|\varphi_g(x)|=\omega_H(g)\leq \ell(g).$$
Now let $\xi\in\operatorname{Dom}(D_H)$. We  prove that
$\pi_q^H(g)\xi\in\operatorname{Dom}(D_H)$. Indeed, for $x\in G/H$,
$$\mathbb L_H(x)(\pi_q^H(g)\xi)(x)=\mathbb L_H(x)\xi(g^{-1}x).$$
Using
$\mathbb L_H(x)=\mathbb L_H(g^{-1}x)+\varphi_g(x)$,
we obtain, for every $x\in G/H$,
$$\mathbb L_H(x)(\pi_q^H(g)\xi)(x)
=\mathbb L_H(g^{-1}x)\xi(g^{-1}x)+\varphi_g(x)\xi(g^{-1}x).$$
Equivalently,
\begin{equation}\label{eq64}
D_H(\pi_q^H(g)\xi)=\pi_q^H(g)(D_H\xi)+M_{\varphi_g}\pi_q^H(g)\xi.
\end{equation}
Indeed, the first term belongs to $\ell^q(G/H)$ because
$D_H\xi\in\ell^q(G/H)$ and $\pi_q^H(g)$ is an isometry on
$\ell^q(G/H)$. The second term belongs to $\ell^q(G/H)$ because
$\pi_q^H(g)\xi\in\ell^q(G/H)$ and $M_{\varphi_g}$ is bounded on
$\ell^q(G/H)$. Hence
$$D_H(\pi_q^H(g)\xi)\in\ell^q(G/H),$$
and therefore
$\pi_q^H(g)\xi\in\operatorname{Dom}(D_H)$.
Moreover, for $\xi\in\operatorname{Dom}(D_H)$, equality \ref{eq64} implies 
$$[D_H,\pi_q^H(g)]\xi
=D_H\pi_q^H(g)\xi-\pi_q^H(g)D_H\xi=M_{\varphi_g}\pi_q^H(g)\xi.$$
Thus
$$[D_H,\pi_q^H(g)]\big|_{\operatorname{Dom}(D_H)}
=M_{\varphi_g}\pi_q^H(g)\big|_{\operatorname{Dom}(D_H)}.$$
Since $M_{\varphi_g}\pi_q^H(g)$ is bounded on $\ell^q(G/H)$, it follows
that
$$\pi_q^H(g)\in\operatorname{Dom}(\delta_H)
\qquad \text{and}\qquad\delta_H(\pi_q^H(g))=M_{\varphi_g}\pi_q^H(g).$$
Next, we claim that for every $n\geq 0$,
$$\pi_q^H(g)\in\operatorname{Dom}(\delta_H^n)\qquad \text{and}\qquad
\delta_H^n(\pi_q^H(g))=M_{\varphi_g^n}\pi_q^H(g), $$
and will prove it by induction. For $n=0$, this is immediate.
Assume the assertion holds for some $n\geq 0$.
Since both $D_H$ and $M_{\varphi_g^n}$ are multiplication operators and
$\varphi_g^n$ is bounded, $M_{\varphi_g^n}$ preserves
$\operatorname{Dom}(D_H)$ and commutes with $D_H$ on this domain.
Hence
$$M_{\varphi_g^n}\in\operatorname{Dom}(\delta_H),
	\qquad\delta_H(M_{\varphi_g^n})=0.$$
Using the Leibniz rule for $\delta_H$, we get
$$\delta_H^{n+1}(\pi_q^H(g))=\delta_H\bigl(M_{\varphi_g^n}\pi_q^H(g)\bigr)
	=M_{\varphi_g^n}\delta_H(\pi_q^H(g))=M_{\varphi_g^{n+1}}\pi_q^H(g).$$
This proves the claim.
Consequently,
\begin{align*}
\|\delta_H^n(\pi_q^H(g))\|_{B(\ell^q(G/H))}
&=\|M_{\varphi_g^n}\pi_q^H(g)\|_{B(\ell^q(G/H))}\\
&\leq\|M_{\varphi_g^n}\|_{B(\ell^q(G/H))}
=\|\varphi_g\|_\infty^n=\omega_H(g)^n.
\end{align*}
Thus, for every finitely supported function  $f$, we obtain 
\begin{align}\label{eq66}
\|\delta_H^n(\pi_q^H(f))\|_{B(\ell^q(G/H))}
&=\left\|\sum_{g\in G} f(g)\delta_H^n(\pi_q^H(g))\right\|_{B(\ell^q(G/H))} \nonumber\\
&\leq\sum_{g\in G}|f(g)|\|\delta_H^n(\pi_q^H(g))\|_{B(\ell^q(G/H))} \nonumber\\
&\leq\sum_{g\in G}|f(g)|\omega_H(g)^n.
\end{align}
Now, if $f\in\ell^1(G)$ and satisfies
$\sum_{g\in G}|f(g)|\omega_H(g)^n<\infty$.
Since $\omega_H(g)^k\leq 1+\omega_H(g)^n$ for $0\leq k\leq n$, we have
$$\sum_{g\in G}|f(g)|\omega_H(g)^k<\infty,\qquad 0\leq k\leq n.$$
Choose an increasing sequence of finite subsets $E_m\subset G$ such that
$$\sum_{g\notin E_m}|f(g)|\omega_H(g)^k\to 0,\qquad 0\leq k\leq n,$$
and set $f_m=f\mathbf \chi_{E_m}$. Then $f_m$ is a sequence of finitely support and 
$$\sum_{g\in G}|f_m(g)-f(g)|\omega_H(g)^k=\sum_{g\notin E_m}|f(g)|\omega_H(g)^k\to 0.$$ 
For each $0\leq k\leq n$, by \eqref{eq66}
$$\|\delta_H^k(\pi_q^H(f_m-f_j))\|_{B(\ell^q(G/H))}
\leq\sum_{g\in G}|f_m(g)-f_j(g)|\omega_H(g)^k.$$
Hence $\delta_H^k(\pi_q^H(f_m))$ is Cauchy in operator norm for each
$0\leq k\leq n$, denote its limit by $T_k$. In particular,
$T_0=\pi_q^H(f)$. Since $\delta_H$ is closed, the convergences
$$\delta_H^k(\pi_q^H(f_m))\to T_k,\qquad \text{and}\qquad 
\delta_H^{k+1}(\pi_q^H(f_m))\to T_{k+1}$$
imply  that
$T_k\in\operatorname{Dom}(\delta_H)$
and $\delta_H(T_k)=T_{k+1}$,
for all $0\leq k<n$. Therefore
$$\pi_q^H(f)=T_0\in\operatorname{Dom}(\delta_H^n)
\qquad \text{and}\qquad \delta_H^n(\pi_q^H(f))=T_n.$$
Finally, by \eqref{eq66},
\begin{align*}
\|\delta_H^n(\pi_q^H(f))\|_{B(\ell^q(G/H))}
&=\|T_n\|_{B(\ell^q(G/H))} \\
&=\lim_{m\to\infty}\|\delta_H^n(\pi_q^H(f_m))\|_{B(\ell^q(G/H))} \\
&\leq\lim_{m\to\infty}\sum_{g\in G}|f_m(g)|\omega_H(g)^n \\
&=\sum_{g\in G}|f(g)|\omega_H(g)^n.
\end{align*}
The proof is completed.
\end{proof}

The following norm estimate is the relative analogue of the estimate of Proposition \ref{prochen1} in Section~4. 

\begin{proposition}\label{prop:pair-gevrey-embedding}
Let $0<\beta<1$ and $1<q<\infty$.  If $f\in\mathcal S_{1,H}^{\infty}(G)$, then $\pi_q^H(f)$ belongs to the relative Gevrey class
$$G_H^{(1/\beta)}(B(\ell^q(G/H))).$$
More precisely, for every $s>0$, with $R_s=(\beta s)^{1/\beta}$,
$$\sup_{n\ge0}\frac{R_s^n\|\delta_H^n(\pi_q^H(f))\|_{B(\ell^q(G/H))}}{(n!)^{1/\beta}}\leq \|f\|_{1,s,H}.$$
\end{proposition}

\begin{proof}
Let $f\in\mathcal S_{1,H}^{\infty}(G)$. For any $s>0$ and $u\ge0$, we have the elementary inequality (see the proof in Proposition \ref{prochen1})
$$u^n\leq (\beta s)^{-n/\beta}(n!)^{1/\beta}e^{su^\beta}.$$
Applying this with $u=\omega_H(g)$, we have 
\begin{align*}
\sum_{g\in G}|f(g)|\omega_H(g)^n 
&\leq (\beta s)^{-n/\beta}(n!)^{1/\beta}\sum_{g\in G}|f(g)|e^{s\omega_H(g)^\beta}\\
&= (\beta s)^{-n/\beta}(n!)^{1/\beta}\|f\|_{1,s,H}.	
\end{align*}
Since $f\in\mathcal S_{1,H}^{\infty}(G)$, the norm $\|f\|_{1,s,H}$ is finite. By Lemma \ref{relative} we obtain
$$\|\delta_H^n(\pi_q^H(f))\|_{B(\ell^q(G/H))} \leq \sum_{g\in G}|f(g)|\omega_H(g)^n.$$
Thus
$$\|\delta_H^n(\pi_q^H(f))\|_{B(\ell^q(G/H))} \leq (\beta s)^{-n/\beta}(n!)^{1/\beta} \|f\|_{1,s,H}.$$
Multiplying both sides by $R_s^n/(n!)^{1/\beta}$ with 
$R_s=(\beta s)^{1/\beta}$ yields
$$\frac{R_s^n\|\delta_H^n(\pi_q^H(f))\|_{B(\ell^q(G/H))}}{(n!)^{1/\beta}} \leq \|f\|_{1,s,H}.$$
Taking the supremum over $n\ge0$ gives
$$\sup_{n\ge0}
\frac{R_s^n\|\delta_H^n(\pi_q^H(f))\|_{B(\ell^q(G/H))}}{(n!)^{1/\beta}}
\leq\|f\|_{1,s,H}.$$
Since $s>0$ is arbitrary and $R_s=(\beta s)^{1/\beta}$ ranges over all
positive numbers, it follows that
$$\pi_q^H(f)\in G_H^{(1/\beta)}(B(\ell^q(G/H))).$$
\end{proof}

Let $PF_q^H(G)$ denote the operator norm closure of $\pi_q^H(\mathbb C[G])$ in $B(\ell^q(G/H))$, and let $\widetilde{PF_q^H(G)}$ be its unitization.  Define the relative  Gevrey-Beurling operator algebra by
$$\mathcal A_{q,\beta}(G,H)=\widetilde{PF_q^H(G)}\cap G_H^{(1/\beta)}(B(\ell^q(G/H))).$$
The main result of this section is the following inverse-closedness
statement.

\begin{theorem}\label{thm:relative-gevrey-wiener}
Let $G$ be a finitely generated discrete group and let $H\leq G$ be a finitely generated subgroup.  Let $0<\beta<1$ and $1<q<\infty$.  If $T\in\mathcal A_{q,\beta}(G,H)$ and $T$ is invertible in $\widetilde{PF_q^H(G)}$, then
$$T^{-1}\in\mathcal A_{q,\beta}(G,H).$$
In particular, if $f\in\mathcal S_{1,H}^{\infty}(G)$ and $\pi_q^H(f)$ is invertible in $\widetilde{PF_q^H(G)}$, then
$$\pi_q^H(f)^{-1}\in\widetilde{PF_q^H(G)}\cap G_H^{(1/\beta)}(B(\ell^q(G/H))).$$
Moreover, the Gevrey seminorms of the inverse satisfy the estimates in Proposition~\ref{prochen2}.
\end{theorem}

\begin{proof}
We apply Proposition~\ref{prochen2} to the Banach algebra
$B(\ell^q(G/H))$ and the closed derivation $\delta_H$.  
Given that $T \in \mathcal A_{q,\beta}(G,H)$ is invertible in $\widetilde{PF_q^H(G)}$, its inverse exists in $B(\ell^q(G/H))$. By Proposition \ref{prochen2}, $T^{-1}$ must also belong to the relative Gevrey class $G_H^{(1/\beta)}(B(\ell^q(G/H)))$. On the other hand,  since $T$ is invertible in $\widetilde{PF_q^H(G)}$, its inverse belongs
to  $\widetilde{PF_q^H(G)}$. Hence  $T^{-1} \in \mathcal A_{q,\beta}(G,H)$.
In particular, for $f \in \mathcal S_{1,H}^{\infty}(G)$, Proposition \ref{prop:pair-gevrey-embedding} ensures $\pi_q^H(f)$ is in the relative Gevrey class. Thus, its inverse automatically falls into $\widetilde{PF_q^H(G)} \cap G_H^{(1/\beta)}(B(\ell^q(G/H)))$. The claimed  estimates for the Gevrey seminorms are then a direct consequence of Proposition~\ref{prochen2} applied to $\delta_H$.
\end{proof}

When $H$ is a normal subgroup of $G$, we immediately obtain the following result.
\begin{corollary}\label{normal}
Let $G$ be a finitely generated discrete group and let $H \triangleleft G$ be a finitely generated normal subgroup. Let $0<\beta<1$ and $1<q<\infty$. Define the quotient weighted weighted Gevrey-Beurling space   by
$$\mathcal S_{1, G/H}^{\infty}(G) = \bigcap_{s>0}\ell^1(G, e^{s \mathbb L_H(gH)^\beta}).$$
If $f \in \mathcal S_{1, G/H}^{\infty}(G)$ and $\pi_q^H(f)$ is invertible in $\widetilde{PF_q^H(G)}$, then
$$\pi_q^H(f)^{-1} \in \widetilde{PF_q^H(G)}\cap G_H^{(1/\beta)}(B(\ell^q(G/H))).$$
\end{corollary}

\begin{proof}
Since $H$ is a normal subgroup, the homogeneous space $G/H$ is a quotient group. The function $$\mathbb L_H(gH)=\inf_{h\in H}\ell(gh)$$ is the quotient length associated with the quotient generating set, satisfying  
$\mathbb L_H(abH)\le\mathbb L_H(aH)+\mathbb L_H(bH)$ and  
$\mathbb L_H(g^{-1}H)=\mathbb L_H(gH)$.
We claim 
$$\omega_H(g)=\sup_{x\in G/H}|\mathbb L_H(x)-\mathbb L_H(g^{-1}x)|$$ coincides exactly with $\mathbb L_H(gH)$. 
For any $aH \in G/H$, 
$$\mathbb L_H(aH)=\mathbb L_H(gH g^{-1}aH)\le\mathbb L_H(gH)+\mathbb L_H(g^{-1}aH)$$
implies $$\mathbb L_H(aH)-\mathbb L_H(g^{-1}aH)\le\mathbb L_H(gH).$$
Similarly, we have
$$\mathbb L_H(g^{-1}aH)=\mathbb L_H(g^{-1}H aH)\le\mathbb L_H(g^{-1}H)+\mathbb L_H(aH)=\mathbb L_H(gH)+\mathbb L_H(aH),$$
which gives 
$$\mathbb L_H(g^{-1}aH)-\mathbb L_H(aH)\le\mathbb L_H(gH).$$
Combining these two estimates, we obtain
$$|\mathbb L_H(aH)-\mathbb L_H(g^{-1}aH)|\le\mathbb L_H(gH)$$
for all $aH\in G/H$. Taking the supremum over all $aH\in G/H$ yields 
$$\omega_H(g)\le\mathbb L_H(gH).$$
On the other hand,  for $eH \in G/H$, we have
$$|\mathbb L_H(eH)-\mathbb L_H(g^{-1}eH)|=|0-\mathbb L_H(g^{-1}H)|=\mathbb L_H(gH).$$
Hence we deduce $\omega_H(g) \ge\mathbb L_H(gH)$. 
Therefore, $\omega_H(g)=\mathbb L_H(gH)$ and the claim is proved. Consequently, the general relative weighted  Gevrey-Beurling space   $\mathcal S_{1,H}^{\infty}(G)$ defined via $\omega_H(g)$ coincides precisely with $\mathcal S_{1, G/H}^{\infty}(G)$. The conclusion then follows immediately as a special case of Theorem~\ref{thm:relative-gevrey-wiener}.
\end{proof}

\begin{example}
Let $G$ be a finitely generated discrete group, and let
$H\triangleleft G$ be a normal subgroup. Set $Q=G/H$.
Then $Q$ is finitely generated. Under the natural identification
$G/H=Q$, the quasi-regular representation of $G$ on $\ell^q(G/H)$
factors through the quotient map $\rho:G\to Q$. More precisely,
for all $g\in G$, $\pi_q^H(g)=\lambda_q(\rho(g))$,
where $\lambda_q$ denotes the left regular representation of $Q$ on
$\ell^q(Q)$. If $\ell$ is the word length on $G$ associated with a finite symmetric generating set $S$, then the quotient length
$\mathbb L_H(gH)=\inf_{h\in H}\ell(gh)$
is precisely the word length on $Q$ associated with the quotient
generating set $\rho(S)$. Moreover, by Corollary~\ref{normal},
$\omega_H(g)=\mathbb L_H(gH)$. Hence the relative weight
$e^{s\omega_H(g)^\beta}$ depends only on the quotient element 
$\rho(g)\in Q$ and coincides with the ordinary  weight
$e^{s|\rho(g)|_Q^\beta}$
on the quotient group. Consequently, after identifying $\ell^q(G/H)$ with $\ell^q(Q)$, the
relative $q$-pseudofunction algebra and the relative Gevrey-Beurling class
associated with $(G,H)$ reduce to the corresponding ordinary objects on
the quotient group $Q$.
Thus, in the normal-subgroup case, the relative theory recovers the
ordinary theory on the quotient group $Q$.
\end{example}

\begin{example}\label{ex612}
The following example illustrates that the relative Gevrey condition
appearing in Proposition~\ref{prop:pair-gevrey-embedding} may be
substantially weaker than the corresponding ordinary Gevrey condition,
because it is controlled by quotient displacement rather than by the
word length on $G$. Let
$$G=F_2\times\mathbb Z^d,\qquad	H=F_2\times\{0\}.$$
Then $H$ is normal in $G$, and $G/H\cong \mathbb Z^d$
via $(u,z)H\mapsto z$. Hence the Schreier graph has polynomial growth,	whereas $G$ has exponential growth because of the $F_2$ factor.
For the standard product generating set, the quotient length is the
standard word length on $\mathbb Z^d$, and for $g=(u,z)$ we have 
	$$\omega_H(u,z)=\sup_{y\in\mathbb Z^d}\big||y|-|y-z|\big|=|z|.$$	
Therefore, for this generating set,
 $f\in\mathcal S_{1,H}^{\infty}(G)$ if and only if
$$\sum_{(u,z)\in F_2\times\mathbb Z^d}|f(u,z)|e^{s|z|^\beta}<\infty
	\qquad\text{for every }s>0.$$
For another finite generating set, the corresponding quotient length is
equivalent to $|z|$, and hence the same weighted weighted Gevrey-Beurling space   is obtained up to
equivalent Fr\'echet seminorms.
This condition imposes Gevrey-type decay only in the quotient direction
$\mathbb Z^d$, and no additional weighted decay in the $F_2$ direction
beyond ordinary summability. By contrast, the ordinary Gevrey condition
would involve the weight
$$e^{s\ell(u,z)^\beta}= e^{s(|u|_{F_2}+|z|)^\beta},$$
which also penalizes the $F_2$ direction. Hence the relative condition
can be substantially weaker than the ordinary condition.
Consequently, Proposition~\ref{prop:pair-gevrey-embedding} applies to
kernels whose Gevrey decay is required only in the quotient direction.
\end{example}

\section{Open problems}

We conclude with several questions and further directions suggested by the
present work. The results above indicate that Gevrey regularity can provide
a useful alternative smoothness framework in settings where rapid decay
estimates are unavailable or too restrictive. This perspective is also
compatible with recent developments beyond the group case. For instance,
work of Austad, Ortega, and Palmstrom on \'{e}tale groupoids with strong
subexponential growth indicates that this growth condition is relevant to
$K$-theoretic invariance phenomena for associated $L^p$-operator algebras
\cite{Aus}. Moreover, recent work on rapid decay for \'{e}tale groupoids
and Fell bundles \cite{Stoiber2025,BusKarmakar2026} suggests that weighted
convolution techniques remain flexible beyond the group setting. The
following questions point to possible extensions of the present framework.

\begin{enumerate}
	\item  Can we prove an intrinsic version of
	Theorem~\ref{inverse-closed} in which all iterated commutators
	$\delta_\ell^n(T)$ belong to $PF_q(G)$, rather than only to
	$B(L^q(G))$? Such a result would make the Gevrey operator algebra an intrinsic smooth subalgebra of $PF_q(G)$.
	
	\item To what extent do the spectral comparison results for compactly
	supported functions extend to weighted Banach convolution algebras
	containing $C_c(G)$ and continuously embedded in $PF_q(G)$?
	
	\item Which group cocycles give rise to cyclic cocycles that extend
	continuously to $\mathcal A_{q,\beta}(G)$? Can such cocycles have
	non-trivial pairings with $K$-theory classes of $PF_q(G)$ via the
	inclusion $\mathcal A_{q,\beta}(G)\hookrightarrow PF_q(G)$?
	
	\item Can $\mathcal A_{q,\beta}(G)$ be related to natural smooth dense
	subalgebras arising in Baum-Connes-type or coarse-geometric
	frameworks, such as assembly maps, coarse completions, or localization
	operator algebras?
	
	\item For an \'{e}tale groupoid equipped with a locally bounded proper
	length function and satisfying a uniform strong subexponential growth
	condition on its source fibres, can one construct weighted Gevrey
	convolution algebras that satisfy spectral comparison or spectral
	invariance in the corresponding reduced $L^q$-operator algebra?
	
	\item Under what hypotheses on a Fell bundle over an \'{e}tale groupoid
	can one construct Gevrey-type weighted section algebras that are
	inverse-closed in the corresponding reduced completion?
\end{enumerate}

\subsection*{Acknowledgement} 
The first author was supported 
by NSFC (No. 12061018). The second author was supported by NSFC (No. 12571135, 12171156), Key Laboratory of MEA (Ministry of Education), the Science and Technology Commission of Shanghai (No. 22DZ2229014), and Shanghai Key Laboratory of PMMP, East China Normal University.

\end{document}